\documentclass[preprint,10pt,3p]{elsarticle}

\usepackage{graphics}
\usepackage{graphicx}

\usepackage{amssymb}
\usepackage{amsthm}
\usepackage{amscd}
\usepackage{amsmath}
\usepackage{amsfonts}
\usepackage{mathtools}
\usepackage{geometry}
\usepackage{longtable}
\usepackage{float}
\usepackage{pdfsync}
\usepackage{xcolor}
\usepackage{booktabs, multirow, caption}
\usepackage[active]{srcltx}
\usepackage[toc, page]{appendix}

\usepackage{breqn}
\usepackage{longtable}
\usepackage{adjustbox}
\usepackage{CJK}
\usepackage{multicol}
\usepackage{listings}
\usepackage{graphicx}
\usepackage{mathptmx}
\usepackage{latexsym}
\usepackage{bm}
\usepackage{enumitem}
\usepackage{indentfirst}
\usepackage{setspace}

\newcommand{\lam}{\lambda}

\newcommand{\R}{\mathbb{R}}

\newcommand{\N}{\mathbb{N}}

\newtheorem{defn}{Definition}[section]
\newtheorem{thm}{Theorem}[section]
\newtheorem{lemma}{Lemma}[section]

\newtheorem{example}{Example}[section]
\newtheorem{remark}{Remark}[section]

\journal{Journal of Computational and Applied Mathematics}

\begin{document}
	
	\begin{frontmatter}
		
		
		\title{An Efficient Numerical Approach for Solving Two-Point Fractional Order Nonlinear Boundary Value Problems with Robin Boundary Conditions}
		\author[label1]{Junseo Lee}
		\ead{junseo.lee@unist.ac.kr}
		
		\author[label1]{Bongsoo Jang}
\ead{bsjang@unist.ac.kr}
\address[label1]{Department of Mathematical Science, Ulsan National Institute of Science and Technology(UNIST), Ulsan 44919, Republic of Korea}

\author[label2]{Hyunju Kim\corref{cor2}}
\ead{hkim@ngu.edu}
\address[label2]{Department of Mathematics, North Greenville University, Tigerville, SC USA 29688}		

\cortext[cor2]{Corresponding author. Tel:+01 864 977 2190}


		\date{\today}
		
		\begin{abstract}
			This article proposes new strategies for solving two-point Fractional order Nonlinear Boundary Value Problems (FNBVPs) with Robin Boundary Conditions (RBCs). In the new numerical schemes, a two-point FNBVP is transformed into a system of Fractional order Initial Value Problems (FIVPs) with unknown Initial Conditions (ICs). To approximate ICs in the system of FIVPs, we develop nonlinear shooting methods based on Newton's method and Halley's method using the RBC at the right end point.
			To deal with FIVPs in a system, we mainly employ High-order Predictor-Corrector Methods (HPCMs) with linear interpolation and quadratic interpolation \cite{nguyen2017high} into Volterra integral equations which are equivalent to FIVPs. 
			The advantage of proposed schemes with HPCMs is that even though they are designed for solving two-point FNBVPs, they can handle both linear and nonlinear two-point Fractional order Boundary Value Problems (FBVPs) with RBCs and have uniform convergence rates of HPCMs, $\mathcal{O}(h^2)$ and $\mathcal{O}(h^3)$ for shooting techniques with Newton's method and Halley's method, respectively.
			A variety of numerical examples are demonstrated to confirm the effectiveness and performance of the proposed schemes. Also we compare the accuracy and performance of our schemes with another method.\\
			
		\end{abstract}
		
		
		
	\end{frontmatter}
	
	\section{Introduction}
	
	Fractional calculus has proven to describe many phenomena in science and engineering more accurately than integer order calculus because of the non-local property of the fractional derivative \cite{Caputo, Dalir, Das, Hilfer, Magin, Podlubny}. Many authors have introduced analytical and numerical methods of the Fractional order Initial Value Problems (FIVPs).
	Relatively, Fractional order Boundary Value Problems (FBVPs) have been paid less attention than FIVPs. Moreover, numerical methods of FNBVPs with Robin Boundary Conditions (RBCs) have hardly been studied.
	In this paper, we consider the two-point Fractional order Nonlinear Boundary Value Problem (FNBVP) with RBCs:
	\begin{equation}\label{eq:1}
	\begin{cases}
	D_{a}^{\alpha_2}y(t) = f(t,y,D_{a}^{\alpha_1}y(t)), \ t\in[a,b],\\
	a_1y(a)+b_1y'(a) = \gamma_1, \ a_2y(b)+b_2y'(b)=\gamma_2,
	\end{cases}
	\end{equation}
	where  $0 < \alpha_1 \le 1, ~1 < \alpha_2 < 2, ~\alpha_1,\alpha_2, \gamma_t \in \R$. $D_{a}^{\alpha_1}$ and $D_{a}^{\alpha_2}$ are Caputo fractional differentiations defined as follows:
	\begin{defn}\label{def-RL}
		Let $\alpha\in \R^{+}$. The operator $J^{\alpha}_{a,t}$, defined on $L_1[a,b]$ by
		\begin{equation*}
		J^{\alpha}_{a}y(t) := \frac{1}{\Gamma(\alpha)}\int_{a}^{t}(t-\tau)^{\alpha-1}y(\tau)\,d\tau,
		\end{equation*}
		for $a \leq t \leq b$, is called the Riemann-Liouville fractional integral operator of order $\alpha$.
	\end{defn}
	For $\alpha = 0 $, $J^{0}_{a} = I$, the identity operator.
	
	\begin{defn} 
		let $\alpha \in \R^{+}$. The operator $D^\alpha_a$ is defined by
		\begin{equation*}
		{}D^{\alpha}_{a}y(t) := J^{\lceil \alpha \rceil-\alpha}_aD^{\lceil \alpha \rceil}y(t) = \frac{1}{\Gamma(\lceil \alpha \rceil - \alpha)}\int_{a}^{t}(t-\tau)^{\lfloor \alpha \rfloor - \alpha}y^{(\lceil \alpha \rceil)}(\tau)d\tau,
		\end{equation*}
		where $\lceil~~\rceil$ is the ceiling function and $\lfloor~~\rfloor$ is the floor function.
	\end{defn}
	The multi-term Caputo sense FIVP can be transformed into the system of FIVPs by the following Theorem \ref{Transform} \cite{diethelm2010analysis}:
	
	\begin{thm}\label{Transform}
		Let us consider the following multi-term Caputo sense fractional differential equation with initial conditions,
		\begin{equation} \label{eq:2}
			\begin{cases}
				D^{\alpha_n}_ay(t) = f(t,y(t),D^{\alpha_1}_ay(t),D^{\alpha_2}_ay(t),\cdots,D^{\alpha_{n-1}}_ay(t)),\\
				y^{(j)}(a) = y^{(j)}_{a},~j = 0,1,\cdots,\lfloor \alpha_n  \rfloor,
			\end{cases}
		\end{equation}
		where $\alpha_n > \alpha_{n-1} > \cdots \alpha_1 > 0,~\alpha_i-\alpha_{i-1} \leq 1$ for all $i=2,3,\cdots,n$ and $0 < \alpha_1 < 1$. Then, we can define $\beta_i$,
		\begin{equation*}
			\begin{cases}
				\beta_1 := \alpha_1,\\
				\beta_i := \alpha_i - \alpha_{i-1}, & i = 2,3,\cdots,n.\\
			\end{cases}
		\end{equation*}
		Then, the multi-term fractional differential equation with initial conditions \eqref{eq:2} is equivalent to the following system of fractional differential equations:
		\begin{subequations}\label{eq:3}
			\begin{equation}\label{eq:3.1}
				\begin{aligned}
					D^{\beta_1}_ay_1(t) &= y_2(t),\\
					D^{\beta_2}_ay_2(t) &= y_3(t),\\
					\vdots\\
					D^{\beta_{n-1}}_ay_{n-1} &= y_n(t),\\
					D^{\beta_{n}}_ay_{n} &= f(t,y_1,y_2,\cdots,y_{n-2},y_{n-1}),
				\end{aligned}
			\end{equation}
			together with the initial conditions:
			\begin{equation}\label{eq:3.2}
				y_i	(t_0) = 
				\begin{cases}
					y^{(0)} & \text{if}~i=1,\\
					y^{(l)} & \text{if}~\alpha_{i-1}=l\in\N,\\
					0 & \text{else},
				\end{cases}
			\end{equation}
		\end{subequations}
		in the following sense:\\
		1. Whenever the function $y \in C^{\lceil \alpha_n \rceil}[a,b]$ is a solution of the multi-term equation with initial conditions \eqref{eq:2}, the vector-valued function $Y := (y_1,\cdots,y_n)^T$ with
		\begin{equation*}
			y_i(t) :=
			\begin{cases}
				y(t) & \text{if}~i=1,\\
				D^{\alpha_{i-1}}_ay(t) & \text{if}~i\neq 1,
			\end{cases}
		\end{equation*}
		is a solution of the system of fractional differential equations \eqref{eq:3.1} with initial conditions \eqref{eq:3.2}\\
		2. Whenever the vector-valued function $Y := (y_1,\cdots,y_n)^T$ is a solution of the system of multi-order fractional differential equations \eqref{eq:3.1} with initial conditions \eqref{eq:3.2}, the function $y:=y_1$ is a solution of the multi-term equation with initial conditions \eqref{eq:2}.
	\end{thm}

	In this paper, we propose new schemes to deal with FNBVPs and the algorithms are summarized as follows:
	\begin{enumerate}
		\item In case of that $0 < \alpha_1 <1$, we transform the FNBVP \eqref{eq:1} with $a=0$ into a system of FIVPs using Theorem \ref{Transform}. \\
		In case of that $\alpha_1 = 1$, i.e. the FNBVP \eqref{eq:1} has a single term of fractional order $\alpha_2$, we substitute the integer order $\alpha_1 = 1$ with the fractional order $\alpha_1 = 1-\varepsilon, \ \varepsilon \rightarrow 0+$ so that the FNBVP satisfies the assumption,  $0 < \alpha_1 < 1$ in Theorem \ref{Transform}. First Gronwall inequality for two-term equations in \cite{diethelm2010analysis} guarantees that the difference between the solution of FNBVP with $\alpha_1=1$ and with $\alpha_1 = 1-\varepsilon$ approaches $0$ as $\varepsilon \rightarrow 0+$. The FNBVP with $\alpha_1 = 1- \varepsilon$ is transformed into a system of FIVPs and then we reduce the number of equations in the system. We prove the reduced system is equivalent to the original system as $\varepsilon \rightarrow 0+$ in Section \ref{sec2}.
		\item To deal with FIVPs, we adopt high-order predictor-corrector methods (HPCMs) with linear interpolation and quadratic interpolation \cite{nguyen2017high} into Volterra integral equations which are equivalent to FIVPs. 
		\item ICs of the FIVPs in the system equivalent to \eqref{eq:1} are obtained by RBC at $t=0$. But ICs include $s:=y'(0)$ and since $s$ is unknown, we approximate $s$ by means of nonlinear shooting techniques based on Newton's method and Halley's method. The error function $|a_2y(b, s) + b_2y'(b, s) - \gamma_2|$ is used to construct the root-finding problem in order to make the approximate solution to $y(t)$ satisfying the RBC at $t=b$.
		\item The algorithm of the proposed shooting technique is as follows: The system of FIVPs is solved with an initial approximation to $s$, $s_k$ at the $k$th iteration. Using the approximate solution to the system obtained by HPCMs with $s_0$, we find $s_1$ by solving Newton's (Halley's) formula. We update the approximate solution to the system with $s_1$ and measure the norm of the error function. We repeat this process until the norm of the error function is within a tolerance.
	\end{enumerate}
	 Similar to our proposed schemes, authors in papers \cite{Huang3,Huang2} introduced numerical methods for solving FBVPs with RBCs. In papers \cite{Huang3, Huang2}, the FBVP with RBCs is turned into the FIVP by using a shooting method with a guess for the unknown IC $y(0)$ and then the FIVP is transformed into the Volterra integral equation. The integral-differential term in the Volterra integral equation is approximated by an integral discretization scheme with constant and first order interpolating polynomials in paper \cite{Huang3} and \cite{Huang2}, respectively. However, the integral discretization schemes can only handle linear FBVPs and the rate of convergence depends on the fractional order $\alpha$. This is elaborately addressed in Chapter \ref{sec4}.
	 The main advantages of our proposed schemes are as follows:
	 \begin{enumerate}
	 	\item The proposed schemes can handle both linear and nonlinear FBVPs with general RBCs. 
	 	\item Our proposed schemes can deal with multi-term FBVPs where $0<\alpha_1\le 1$ and $1<\alpha_2<2$. 
	 	\item 
	 	Our proposed methods with HPCMs have uniform convergence rates $\mathcal{O}(h^2)$ and $\mathcal{O}(h^3)$ for shooting techniques based on Newton's method and Halley's method respectively, with enough iterations, regardless of fractional orders thanks to global error estimates of HPCMs in \cite{nguyen2017high}.
	 	\item A matrix equation established iteratively is not involved as Newton's method and Halley's method are applied into a system of FIVPs.
	 \end{enumerate}

	This article is organized as follows. In Section \ref{sec2}, we describe an idea about the transformation of FNBVP with RBCs \eqref{eq:1} into a system of FIVPs. In Section \ref{sec3}, we describe nonlinear shooting methods based on Newton's method and Halley's method, to approximate unknown IC $s:=y(0)$ of FIVPs in the system. Also, we briefly mention how to apply the HPCMs into a system of FIVPs in Section \ref{sec3}. In Section \ref{sec4}, We demonstrate numerical examples verifying that proposed shooting techniques combined with HPCMs guarantee the global convergence rates of HPCMs. We also confirm the performance and effectiveness of the proposed methods by comparing with the modified integral discretization scheme in the paper \cite{Huang2}. A conclusion will be given in Section \ref{sec5}.
	Finally tables of numerical results and the linear explicit method which is an alternative method for solving FIVPs are described in the Appendix.

	\section{Transformation of FNBVP into system of FIVPs}\label{sec2}
	In this section, we describe how to transform the FNBVP with RBCs \eqref{eq:1} into a system of FIVPs according to the value of $\alpha_1$. Basically, we apply Theorem \ref{Transform} with $\beta_1:=\alpha_1, \ \beta_2:=1-\alpha_1, \ \beta_3:=\alpha_2-1$ to the FNBVP in case of that $0<\alpha_1<1$. If $\alpha_1$ is equal to 1, then we replace $\alpha_1$ with $1-\epsilon, \ \epsilon \rightarrow 0+$ and set $\beta_1:=1-\epsilon, \ \beta_2:=\epsilon, \ \beta_3:=\alpha_2 - 1$. We reduce the size of system using the fact $\beta_2 \rightarrow 0+$.\\
	
	


	\noindent Case 1: $0<\alpha_1<1$\\
	First, we consider a FNBVP with Dirichlet boundary conditions as follows:
	\begin{equation}\label{eq:4}
		\begin{cases}
			D^{\alpha_2}_0y(t) = f(t,y(t),D^{\alpha_1}_0y(t)), \ t\in[0,b] \\
			y(0) = y_{0}, \ y(T) = y_b.
		\end{cases}	
	\end{equation}
	 Applying Theorem \ref{Transform} with $\beta_1:=\alpha_1, \ \beta_2:=1-\alpha_1, \ \beta_3:=\alpha_2-1$,  FNBVP with Dirichlet boundary conditions \eqref{eq:4} can be transformed as follows:
	\begin{equation}\label{eq:5}
		\begin{cases}
			D^{\alpha_1}_0y(t) = w(t), & ~y(0) = y_0,\,y(b) = y_b,\\
			D^{1-\alpha_1}_0w(t) = z(t), & ~w(0) = 0,\\
			D^{\alpha_2-1}_0z(t) = f(t,y(t),w(t)), & ~z(0) = y^{(1)}(0).
		\end{cases}
	\end{equation}
	From the system of fractional differential equations \eqref{eq:5}, we obtain the following system of FIVPs:
	\begin{equation}\label{T_DIVP}
		\begin{cases}
			D^{\alpha_1}_0y(t) = w(t), & y(0) = y_0,\\
			D^{1-\alpha_1}_0w(t) = z(t), & w(0) = 0,\\
			D^{\alpha_2-1}_0z(t) = f(t,y(t),w(t)), & z(0) = s,
		\end{cases}
	\end{equation}
	where the IC $s$ is unknown and so needs to be approximated.\\
	Similar to the case of Dirichlet boundary conditions, FNBVP with RBCs \eqref{eq:1} can be written as follows:
	\begin{equation}\label{T_RIVP}
		\begin{cases}
			D^{\alpha_1}_0y(t) = w(t), & y(0) = y_0 = \frac{\gamma_1-b_1s}{a_1}\\
			D^{1-\alpha_1}_0w(t) = z(t), & w(0) = 0\\
			D^{\alpha_2-1}_0z(t) = f(t,y(t),w(t)), & z(0) = s.
		\end{cases}
	\end{equation}
	
	\noindent Case 2: $\alpha_1 = 1$ \\
	We consider the following FBVP with Dirichlet boundary conditions:
	\begin{equation}\label{FDE_ODE}
		\begin{cases}
			D^{\alpha_2}_0y(t) = f(t,y(t),y'(t)),~~t\in[0,b],\\
			y(0)=y_0, \ y(b)=y_b,
		\end{cases}
	\end{equation}
	where  $1 < \alpha_2 < 2,~\alpha_2 \in \R$.
	Since the fractional differential equation in \eqref{FDE_ODE} does not satisfy the assumption, $0<\alpha_1<1$ in Theorem \ref{Transform}, 
	we cannot apply the strategy used in \eqref{eq:4} to \eqref{FDE_ODE}.
	So we modify the equation in \eqref{FDE_ODE} to meet the assumption, with the same boundary conditions as in \eqref{FDE_ODE} as follows:
	\begin{equation}\label{FDE_ODE-ver1}
		\begin{cases}
			D^{\alpha_2}_0y(t) = f(t,y(t), D^{1-\varepsilon}_0y(t)), ~~t\in[0,b],\\
			y(0)=y_0, \ y(b)=y_b,
		\end{cases}
	\end{equation}
	where  $\alpha_2 \in (1,2), ~\epsilon \rightarrow 0+$. By Lemma \ref{Lemma_of_Gronwall_ineq}, solutions of two FBVPs \eqref{FDE_ODE} and \eqref{FDE_ODE-ver1} are approximately equal and the absolute error depends on $\epsilon$.
	
	
	
	\begin{lemma}\textnormal{(First Gronwall inequality for two-term equations in \cite{diethelm2010analysis}).}\label{Lemma_of_Gronwall_ineq}
		Let $\alpha_2 > 0$ and $\alpha_1,\tilde{\alpha}_1 \in (0,\alpha_2)$ be chosen so that the equation
		\begin{equation*}
			D^{\alpha_2}_{0}y(t) = f(t,y(t),D^{\alpha_1}_{0}y(t)),
		\end{equation*}
		subject to the initial conditions
		\begin{equation*}
			y(0) = y_0,y^{(1)}(0)=y^{1}_0,\cdots,y^{(\lceil \alpha_2 \rceil-1)}(0)=y^{\lceil \alpha_2 \rceil-1}_0
		\end{equation*}
		and
		\begin{equation*}
			D^{\alpha_2}_{0}z(t) = f(t,z(t),D^{\tilde{\alpha}_1}_0\tilde{z}(t))
		\end{equation*}
		subject to the same initial conditions
		\begin{equation*}
			z(0) = y_0,z^{(1)}(0)=y^{1}_0,\cdots,z^{(\lceil \alpha_2 \rceil-1)}(0)=y^{\lceil \alpha_2 \rceil-1}_0
		\end{equation*}
		(where $f$ satisfies a Lipschitz condition in its second and third arguments on a suitable domain) have unique continuous solutions $y,z:[0,T]\rightarrow\R$. We assume further that $\lfloor\alpha_1\rfloor = \lfloor\tilde{\alpha_1}\rfloor$. Then there exist constants $K_1$ and $K_2$ such that
		\begin{equation*}
			\vert y(t)-z(t) \vert \leq K_1\vert \alpha_1 -\tilde{\alpha}_1 \vert E_{\alpha_n}(K_2T^{\alpha_2}),~~^{\forall} t \in [0,T],
		\end{equation*}
		where $E_{\alpha_n}$ denotes the Mittag-Leffler function of order $\alpha_n$.
	\end{lemma}
	
	Applying Theorem \ref{Transform} to \eqref{FDE_ODE-ver1}, we obtain the following system of FIVPs:
	\begin{equation}\label{FDE_ODE-ver2}
		\begin{cases}
			D^{1-\epsilon}_0y(t) = w(t), & y(0) = y_0,\\
			D^{\epsilon}_0w(t) = z(t), & w(0) = 0,\\
			D^{\alpha_2-1}_0z(t) = f(t,y(t),w(t)), & z(0)  = s.
		\end{cases}
	\end{equation}
	Now, we show that the system of FIVPs \eqref{FDE_ODE-ver2} is equivalent to the following system as $\epsilon \rightarrow 0$ using Lemmas \ref{lem.1} through \ref{order_com_lemma} and Theorem \ref{epsapprox}.
	\begin{equation}\label{FDE_ODE-ver3}
		\begin{cases}
			D^{1-\epsilon}_0\tilde{y}(t) = \tilde{z}(t), & \tilde{y}(0) = y_0,\\
			D^{\alpha_2-1}_0\tilde{z}(t) = f(t,\tilde{y}(t),\tilde{z}(t)). & \tilde{z}(0) = s.
		\end{cases}
	\end{equation}
	\begin{lemma}(Theorem 2.10 in \cite{diethelm2010analysis})\label{lem.1}
		Let $f\in C[a,b]$ and $m\geq 0$. Moreover assume that $\alpha_k$ is a sequence of positive numbers such that $\alpha_k \rightarrow \alpha$ as $k \rightarrow \infty$. Then, for every $\delta>0$,
		\begin{equation*}
		\lim_{k\rightarrow\infty}\sup_{t\in[a+\delta,b]}\vert J^{\alpha_k}_af(t) - J^{\alpha}_af(t) \vert = 0.
		\end{equation*}
	\end{lemma}
	\begin{lemma}(Lemma 6.19 in \cite{diethelm2010analysis})\label{lem.2}
		Let $\alpha, \ T, \ \phi_1, \phi_2\in\mathbb{R}^{+}$. Moreover assume that $\delta \ : \ [0,T]\rightarrow\mathbb{R}$ is a continuous function satisfying the inequality
		\begin{equation*}
			\vert\delta(t)\vert\le\phi_1 + \frac{\phi_2}{\Gamma(\alpha)}\int^t_0(t-\tau)^{\alpha-1}\vert\delta(\tau)\vert d\tau,~~ \forall t\in[0,T].
		\end{equation*}
		Then 
		\begin{equation*}
			\vert\delta(t)\vert\le\phi_1 E_{\alpha}(\phi_2 t^{\alpha}), ~~ \forall t\in[0,T].
		\end{equation*} 
	\end{lemma}
	\begin{lemma}\label{order_com_lemma}
		Let $0<\gamma \leq \alpha \leq \beta $. Then for any $t\in[a,b]$,
		\begin{equation*}
			\vert J^{\alpha}_ay(t)-J^{\beta}_ay(t) \vert \leq \left[ \frac{\Gamma(\gamma)}{\Gamma(\alpha)}(b-a)^{\alpha-\gamma} + \frac{\Gamma(\gamma)}{\Gamma(\beta)}(b-a)^{\beta-\gamma} \right] J^{\gamma}_a\vert y(t)\vert.
		\end{equation*}
	\end{lemma}
		\begin{proof}
		By Definition \ref{def-RL}, 
		\begin{equation*}
		\begin{aligned}
		J^{\alpha}_ay(t) + J^{\beta}_ay(t) &= \frac{1}{\Gamma(\alpha)}\int_a^t(t-\tau)^{\alpha-1}y(\tau)d\tau + \frac{1}{\Gamma(\beta)}\int_a^t(t-\tau)^{\beta-1}y(\tau)d\tau\\
		&= \frac{1}{\Gamma(\alpha)}\int_a^t(t-\tau)^{\gamma-1+(\alpha-\gamma)}y(\tau)d\tau + \frac{1}{\Gamma(\beta)}\int_a^t(t-\tau)^{\gamma-1+(\beta-\gamma)}y(\tau)ds.
		\end{aligned}
		\end{equation*}
		Then,
		\begin{equation*}
		\begin{aligned}
		\vert J^{\alpha}_ay(t)-J^{\beta}_ay(t) &\vert \leq \frac{(b-a)^{\alpha-\gamma}}{\Gamma(\alpha)}\int_a^t(t-\tau)^{\gamma-1}\vert y(\tau)\vert d\tau + \frac{(b-a)^{\beta-\gamma}}{\Gamma(\beta)}\int_a^t(t-\tau)^{\gamma - 1}\vert y(\tau) \vert d\tau\\
		&\leq \left[\frac{\Gamma(\gamma)}{\Gamma(\alpha)}(b-a)^{\alpha-\gamma}+\frac{\Gamma(\gamma)}{\Gamma(\beta)}(b-a)^{\beta-\gamma}\right]\frac{1}{\Gamma(\gamma)}\int_a^t(t-\tau)^{\gamma-1}\vert y(\tau) \vert d\tau\\
		&\leq  \left[\frac{\Gamma(\gamma)}{\Gamma(\alpha)}(b-a)^{\alpha-\gamma}+\frac{\Gamma(\gamma)}{\Gamma(\beta)}(b-a)^{\beta-\gamma}\right] J^{\gamma}_a\vert y(t) \vert.
		\end{aligned}
		\end{equation*}
	\end{proof}
	
	
	\begin{thm}\label{epsapprox}
		Let $1<\alpha_2<2$, $^{\forall} T\in\R^+$ and $f:[0,T]\times \R \times \R \rightarrow \R$ satisfies the  Lipsichiz condition in its second and third arguments on a suitable domain. Then we have the following inequality:
		\begin{equation*}
		\vert f(t,x_1,y_1)-f(t,x_2,y_2)\vert \leq L\left( \vert x_2-x_1 \vert + \vert y_2 - y_1 \vert \right), 
		\end{equation*}
		where $^{\forall} t \in [0,T]$, $x_1,x_2,y_1,y_2 : [0,T] \rightarrow \R$ and $0<L$. \\
		If for any $ 0 <\epsilon << 1$, $\hat{y}$ and $\tilde{y}$ are solutions of the following systems, respectively:
		\begin{equation*}
			\begin{cases}
				D^{1-\epsilon}_0\hat{y}(t) = w(t), & y(0) = y_0\\
				D^{\epsilon}_0w(t) = z(t), & w(0) = 0\\
				D^{\alpha_2-1}_0z(t) = f(t,\hat{y}(t),w(t)). & z(0) = s
			\end{cases}
			~~
			\text{and}
			~~
			\begin{cases}
				D^{1-\epsilon}_0\tilde{y}(t) = \tilde{z}(t), & \tilde{y}(0) = y_0\\
				D^{\alpha_2-1}_0\tilde{z}(t) = f(t,\tilde{y}(t),\tilde{z}(t)). & \tilde{z}(0) = s,
			\end{cases}
		\end{equation*}
		then,
		\begin{equation*}
			\vert \hat{y}(t)-\tilde{y}(t)\vert \rightarrow 0, ~~\text{as}~\epsilon \rightarrow 0.
		\end{equation*}
	\end{thm}
	
	\begin{proof}
		By Lemma 6.2 in \cite{diethelm2010analysis}, FIVPs $D^{1-\epsilon}_0\hat{y}(t) = w(t), \ y(0) = y_0, \text{ and } D^{1-\epsilon}_0\tilde{y}(t) = \tilde{z}(t), \ \tilde{y}(0) = y_0$ 
		are equivalent to Volterral integral equations of the second kind, respectively as follows:
		\begin{equation*}
		\hat{y}(t) = y_0+J^{1-\epsilon}_0w(t),~~\tilde{y}(t) = y_0+J^{1-\epsilon}_0\tilde{z}(t).
		\end{equation*}
		Then $\hat{y}(t) - \tilde{y}(t)$ can be expressed as Riemann-Liouville fractional integral of $w(t)-\tilde{z}(t)$ as follows:
		\begin{equation}\label{diff_y_y}
		\hat{y}(t) - \tilde{y}(t) = J^{1-\epsilon}_0(w(t)-\tilde{z}(t)).
		\end{equation}
		Since $w(t) = J^{\epsilon}z(t)$ by Lemma 6.2 in \cite{diethelm2010analysis} and rewriting $w(t) - \tilde{z}(t)$ as $w(t)-J^{\epsilon}_0\tilde{z}(t)+J^{\epsilon}_0\tilde{z}(t)-\tilde{z}(t)$, we obtain the following inequality:
		\begin{equation}\label{diff_w_z}
		\vert w(t) - \tilde{z}(t) \vert \leq  J^{\epsilon}_0\vert z(t)-\tilde{z}(t) \vert + \vert J^{\epsilon}_0\tilde{z}(t)-\tilde{z}(t) \vert.
		\end{equation}
		Since $D^{\alpha_2-1}_0z(t) = f(t,\hat{y}(t),w(t))$ and $D^{\alpha_2-1}_0\tilde{z}(t) = f(t,\tilde{y}(t),\tilde{z}(t))$ are equivalent to Volterra integral equations $z(t) = s+ J^{\alpha_2-1}_0f(t,\hat{y}(t),w(t))$ and $\tilde{z}(t) = s + J^{\alpha_2-1}_0f(t,\tilde{y}(t),w(t))$ by Lemma 6.2 in \cite{diethelm2010analysis} respectively, using Lipschitz condition, \eqref{diff_y_y}, and \eqref{diff_w_z}, we obtain the following inequalities:
		\begin{eqnarray}\label{lem2.4-pf-ineq1}
		\vert z(t)-\tilde{z}(t) \vert &=& \vert J^{\alpha_2-1}_0\left[f(t,\hat{y}(t),w(t))-f(t,\tilde{y}(t),\tilde{z}(t))\right] \vert \nonumber \\
		&\leq& J^{\alpha_2-1}_0\vert f(t,\hat{y}(t),w(t)) - f(t,\tilde{y}(t),\tilde{z}(t)) \vert \nonumber \\
		&\leq& LJ^{\alpha_2-1}_0 \left( \vert \hat{y}(t)-\tilde{y}(t) \vert + \vert w(t) - \tilde{z}(t) \vert \right) \hspace{2.35cm}\nonumber \\
		&\leq& L J^{\alpha_2-1}_0\left[J^{1-\epsilon}_0\vert w(t)-\tilde{z}(t) \vert +  J^{\epsilon}_0\vert z(t)-\tilde{z}(t) \vert + \vert J^{\epsilon}_0\tilde{z}(t)-\tilde{z}(t)\vert\right].
		\end{eqnarray}
		Since $1<\alpha_2 <2$ and $\Gamma(\alpha_2) = (\alpha_2 - 1)\Gamma(\alpha_2 - 1)$, we have the following inequalities about $J^{\alpha_2 - 1}_0\vert J^{\epsilon} \tilde{z}(t) - \tilde{z}(t)\vert$:
		\begin{eqnarray}\label{lem2.4-pf-ineq2}
		J^{\alpha_2 - 1}_0\vert J^{\epsilon} \tilde{z}(t) - \tilde{z}(t)\vert &=& \frac{1}{\Gamma(\alpha_2 - 1)}\int_0^t (t-\tau)^{\alpha_2 - 2}\vert J^{\epsilon}_0\tilde{z}(t) - \tilde{z}(t)\vert d\tau \nonumber \\
		&\leq& \frac{T^{\alpha_2 - 1}}{\Gamma(\alpha_2 - 1)} \Vert J^{\epsilon}_0\tilde{z}(t) - \tilde{z}(t)\Vert_{\infty} \nonumber \\
		&\leq& \frac{T^{\alpha_2 - 1}}{\Gamma(\alpha_2)} \Vert J^{\epsilon}_0\tilde{z}(t) - \tilde{z}(t)\Vert_{\infty}.
		\end{eqnarray}
		Similarly, we can obtain the following inequality:
		\begin{equation}\label{lem2.4-pf-ineq3}
		J^{\alpha_2 - \epsilon}_0\vert J^{\epsilon}_0 \tilde{z}(t) - \tilde{z}(t)\vert \leq \frac{T^{\alpha_2 - \epsilon}}{\Gamma(\alpha_2 - \epsilon + 1)}\Vert J^{\epsilon}_0\tilde{z}(t) - \tilde{z}(t)\Vert_{\infty}.
		\end{equation}
		Applying the inequality \eqref{lem2.4-pf-ineq2} into \eqref{lem2.4-pf-ineq1}, we have the following inequality:
		\begin{equation}\label{lem2.4-pf-ineq4}
		\vert z(t)-\tilde{z}(t) \vert \leq L \left[ J^{\alpha_2-\epsilon}_0 \vert w(t)-\tilde{z}(t) \vert + J^{\alpha_2-1+\epsilon}_0 \vert z(t)-\tilde{z}(t) \vert + \frac{T^{\alpha_2-1}}{\Gamma(\alpha_2)}\Vert J^{\epsilon}_0\tilde{z}(t) - \tilde{z}(t) \Vert_{\infty} \right].
		\end{equation}
		Using the inequalities \eqref{diff_w_z} and \eqref{lem2.4-pf-ineq3}, we have the following inequalities:
		\begin{eqnarray} \label{ast}
		J^{\alpha_2-\epsilon}_0\vert w(t) - \tilde{z}(t) \vert  &\leq& J^{\alpha_2-\epsilon}_0\left( J^{\epsilon}_0\vert z(t) - \tilde{z}(t) \vert +  \vert J^{\epsilon}_0\tilde{z}(t)-\tilde{z}(t) \vert \right) \nonumber \\
		&\leq& J^{\alpha_2}_0\vert z(t) - \tilde{z}(t) \vert + \frac{T^{\alpha_2-\epsilon}}{\Gamma(\alpha_2-\epsilon+1)}\Vert J^{\epsilon}_0 \tilde{z}(t)-\tilde{z}(t) \Vert_{\infty}.
		\end{eqnarray}
		Applying the inequality \eqref{ast} into \eqref{lem2.4-pf-ineq4} and using Lemma \ref{order_com_lemma} with $\gamma=\alpha_2 - 1, \ \alpha = \alpha_2 - 1 + \epsilon, \ \beta = \alpha_2$, we have the following inequalities:
		\begin{equation*}
		\begin{aligned}
		\vert z(t)-\tilde{z}(t) \vert &\leq L \bigg[ J^{\alpha_2}_0 \vert z(t)-\tilde{z}(t) \vert + \frac{T^{\alpha_2-\epsilon}}{\Gamma(\alpha_2-\epsilon+1)}\Vert J^{\epsilon}_0\tilde{z}(t) - \tilde{z}(t) \Vert_{\infty}\\
		& \hspace{4cm} + J^{\alpha_2-1+\epsilon}_0 \vert z(t) - \tilde{z}(t) \vert + \frac{T^{\alpha_2-1}}{\Gamma(\alpha_2)} \Vert J^{\epsilon}_0\tilde{z}(t) - \tilde{z}(t) \Vert_{\infty} \bigg]\\
		&= L\bigg[ J^{\alpha_2-\epsilon}_0 \vert z(t)-\tilde{z}(t) \vert + J^{\alpha_2-1+\epsilon}_0\vert z(t)-\tilde{z}(t) \vert\\
		& \hspace{4cm} +\left\{ \frac{T^{\alpha_2-\epsilon}}{\Gamma(\alpha_2-\epsilon+1)} + \frac{T^{\alpha_2-1}}{\Gamma(\alpha_2)} \right\} \Vert J^{\epsilon}_0\tilde{z}(t) - \tilde{z}(t) \Vert_{\infty} \bigg]\\
		&\leq L\bigg[ \left\{  \frac{\Gamma(\alpha_2-1)}{\Gamma(\alpha_2-1+\epsilon)}T^{\epsilon} + \frac{\Gamma(\alpha_2-1)}{\Gamma(\alpha_2)}T \right\} J^{\alpha_2-1}_0 \vert z(t)-\tilde{z}(t) \vert\\
		&\hspace{4cm} + \left\{ \frac{T^{\alpha_2-\epsilon}}{\Gamma(\alpha_2-\epsilon+1)} + \frac{T^{\alpha_2-1}}{\Gamma(\alpha_2)} \right\} \Vert J^{\epsilon}_0\tilde{z}(t) - \tilde{z}(t) \Vert_{\infty} \bigg],
		\end{aligned}
		\end{equation*} 
		so that 
		\begin{equation}\label{lem2.4-pf-ineq5}
		\vert z(t) - \tilde{z}(t) \vert \leq LC^1_{\epsilon}J^{\alpha_2-1}_0\vert z(t)-\tilde{z}(t) \vert + LC^2_{\epsilon}\Vert J^{\epsilon}_0 \tilde{z}(t) - \tilde{z}(t) \Vert_{\infty},
		\end{equation}
		where 
		\begin{equation*}
		C^1_{\epsilon} \equiv  \frac{\Gamma(\alpha_2-1)}{\Gamma(\alpha_2-1+\epsilon)}T^{\epsilon} + \frac{\Gamma(\alpha_2-1)}{\Gamma(\alpha_2)}T,~~C^2_{\epsilon} \equiv \frac{T^{\alpha_2-\epsilon}}{\Gamma(\alpha_2-\epsilon+1)}+\frac{T^{\alpha_2-1}}{\Gamma(\alpha_2)}.
		\end{equation*}
		Applying Lemma \ref{lem.2} into inequality \eqref{lem2.4-pf-ineq5}, we obtain the following inequality:
		\begin{equation*}
		\vert z(t) - \tilde{z}(t) \vert \leq LC^2_{\epsilon}\Vert J^{\epsilon}_0\tilde{z}(t)-\tilde{z}(t) \Vert_{\infty}E_{\alpha_2-1}[LC^1_{\epsilon}T^{\alpha_2-1}].
		\end{equation*}
		Therefore,
		\begin{equation*}
		\begin{aligned}
		\vert \tilde{y}(t) - \hat{y}(t) \vert &\leq J^{1-\epsilon}_0\vert w(t) - \tilde{z}(t) \vert\\
		&\leq J^1_0\vert z(t)-\tilde{z}(t) \vert + J^{1-\epsilon}_0\vert J^{\epsilon}_0 \tilde{z}(t) - \tilde{z}(t) \vert\\
		&\leq T\vert z(t)-\tilde{z}(t) \vert + \frac{T^{1-\epsilon}}{\Gamma(2-\epsilon)}\Vert J^{\epsilon}_0 \tilde{z}(t) - \tilde{z}(t) \Vert_{\infty}\\
		&\leq TLC^2_{\epsilon}\Vert J^{\epsilon}_0\tilde{z}(t)-\tilde{z}(t) \Vert_{\infty}E_{\alpha_2-1}[LC^1_{\epsilon}T^{\alpha_2-1}]+ \frac{T^{1-\epsilon}}{\Gamma(2-\epsilon)}\Vert J^{\epsilon}_0 \tilde{z}(t) - \tilde{z}(t) \Vert_{\infty}\\
		&=\left[ TLC^2_{\epsilon}E_{\alpha_2-1}[LC^1_{\epsilon}T^{\alpha_2-1}] + \frac{T^{1-\epsilon}}{\Gamma(2-\epsilon)} \right] \Vert J^{\epsilon}_0 \tilde{z}(t) - \tilde{z}(t) \Vert_{\infty}.
		\end{aligned}
		\end{equation*}
		$C_{\epsilon} \equiv TLC^2_{\epsilon}E_{\alpha_2-1}[LC^1_{\epsilon}T^{\alpha_2-1}] + \frac{T^{1-\epsilon}}{\Gamma(2-\epsilon)}.$ Then,
		\begin{equation*}
		\vert \tilde{y}(t) - \hat{y}(t) \vert  \leq C_{\epsilon}\Vert J^{\epsilon}_0 \tilde{z}(t) - \tilde{z}(t)\Vert_{\infty}.
		\end{equation*}
		Thus, by Lemma \ref{lem.1},
		\begin{equation*}
		\vert \tilde{y}(t) - \hat{y}(t) \vert  \leq C_{\epsilon}\Vert J^{\epsilon}_0 \tilde{z}(t) - \tilde{z}(t)\Vert_{\infty} \rightarrow 0,~\text{as}~\epsilon \rightarrow 0.
		\end{equation*}
	\end{proof}

	For the FNBVP with RBCs \eqref{eq:1}, similarly, by Lemma \ref{Lemma_of_Gronwall_ineq} and using Theorem \ref{Transform}, we obtain the following FIVPs:
	\begin{equation}\label{FDE_ODE-ver2-RBC}
		\begin{cases}
			D^{1-\epsilon}_0y(t) = w(t), & y(0) = \frac{\gamma_1-b_1s}{a_1},\\
			D^{\epsilon}_0w(t) = z(t), & w(0) = 0,\\
			D^{\alpha_2-1}_0z(t) = f(t,y(t),w(t)), & z(0)  = s.
		\end{cases}
	\end{equation}
	The system \eqref{FDE_ODE-ver2-RBC} can be reduced as follows:
	\begin{equation}\label{FDE_ODE-ver3-RBC}
	\begin{cases}
	D^{1-\epsilon}_0\tilde{y}(t) = \tilde{z}(t), & \tilde{y}(0) = \frac{\gamma_1-b_1s}{a_1},\\
	D^{\alpha_2-1}_0\tilde{z}(t) = f(t,\tilde{y}(t),\tilde{z}(t)). & \tilde{z}(0) = s.
	\end{cases}
	\end{equation}
	\section{Nonlinear Shooting Methods and High-Order Predictor-Corrector Methods}\label{sec3}
	FBVPs have been transformed to systems of FIVPs in Section \ref{sec2}.
	Before we address how to deal with systems of FIVPs \eqref{T_DIVP}, \eqref{T_RIVP}, \eqref{FDE_ODE-ver3}, and \eqref{FDE_ODE-ver3-RBC} using High-order Predictor-Corrector Methods (HPCMs), the unknown IC $z(0)=s$ should be handled first. In this section, we describe two nonlinear shooting techniques based on Newton's method and Halley's method to approximate $s$. Both Newton's formula and Halley's formula are designed to determine the solution of a system of FIVPs satisfying the RBC at the right end point of an interval. Without loss of generality, we consider the system of FIVPs \eqref{T_RIVP} that is equivalent to the FNBVP with RBCs \eqref{eq:1}.
	
	In order that the RBC at the right end point $a_2y(b) + b_2y'(b)=\gamma_2$ is involved in approximating $s$, we define  $y(s) := y(s,t)\vert_{t=b}$ and let the error function be $F(s) := a_2y(s)+b_2\frac{\partial}{\partial t}y(s)-\gamma_2$. We approximate the solution of the  root-finding problem $F(s)=0$ by using Newton's method and Halley's method, respectively. For convenience, we denote $$ g_s(t) = \frac{\partial g(s,t)}{\partial s},~~g_{ss} = \frac{\partial^2 g(s,t)}{\partial s^2}$$ throughout this section.
	
	
	
	\subsection{Shooting with Newton's Method}
	The conventional Newton's formula for $F(s)=0$ can be expressed as follows:
	\begin{equation}\label{newton}
	s_{k+1} = s_k-\frac{F(s_k)}{F_s(s_k)},~~~k=0,1,2,\ldots, m,
	\end{equation}
	where $m$ is the maximum number of iterations and 
	\begin{eqnarray}\label{newton-fs}
		F_s(s_k)&=&\frac{\partial F}{\partial s}(s)\vert_{s=s_k, t=b} \nonumber \\
		&=& a_2 \frac{\partial y}{\partial s}(s,t)\vert_{s=s_k, t=b} ~ + ~ b_2\frac{\partial}{\partial s}\left[\frac{\partial y}{\partial t}(s,t)\right]\Bigl{|}_{s=s_k, t=b}.
	\end{eqnarray}
	Observing $y_s(s_k)$ and $y_{ts}(s_k)$, it turns out that they are equal to $\frac{\partial}{\partial s}y(t)\vert_{s=s_k, t=b} \text{ and } \frac{\partial}{\partial s}z(t)\vert_{s=s_k, t=b}$, respectively in the system of FIVPs \eqref{T_RIVP}. Thus we solve the following system obtained from system of FIVPs \eqref{T_RIVP} by applying the operator $\frac{\partial}{\partial s}$ using HPCMs for each $k$:
	\begin{equation}\label{Newton-system}
		\begin{cases}
			D^{\alpha_1}_0y_s(t) = w_s(t), & y_s(0) = -b_1/a_1,\\
			D^{1-\alpha_1}_0w_s(t) = z_s(t), & w_s(0) = 0,\\
			D^{\alpha_2-1}_0z_s(t) = f_s(t,y(t),w(t)), & z_s(0) = 1.
		\end{cases}
	\end{equation}
	Since both $t$ and $s$ are independent variables, $f_s(t,y(t),w(t))$ can be written as
	\begin{equation*}
	f_s(t,y(t),w(t)) = f_y\cdot y_s (t) + f_w \cdot w_s (t).
	\end{equation*}
	The detailed description of HPCMs dealing with a system of FIVPs is in Subsection \ref{sec3-3}. By solving the system \eqref{Newton-system}, $s_{k+1}$ in the Newton's formula \eqref{newton} is computed. Using the updated approximate value of IC $s$, $s_{k+1}$, we update approximate solutions of systems of FIVPs \eqref{T_DIVP}, \eqref{T_RIVP}, \eqref{FDE_ODE-ver3}, and \eqref{FDE_ODE-ver3-RBC}. Repeating this process, we obtain an $s_k$ having an acceptable error of the root-finding problem $F(s)=0$ at an appropriate number of iterations $k$.

	\subsection{Shooting with Halley's Method }
	The conventional Halley's formula for $F(s)=0$ is as follows:
	\begin{equation*}
	s_{k+1} = s_k - \frac{2F(s_k)F_s(s_k)}{2F_s^{2}(s_k)-F(s_k)F_{ss}(s_k)},~~~k=0,1,2,\ldots, m,
	\end{equation*}
	where $F_s(s_k)$ is described in \eqref{newton-fs} and 
	\begin{eqnarray*}
	F_{ss}(s_k)&=&\frac{\partial^2 F}{\partial s^2}(s)\vert_{s=s_k, t=b}\\
	&=& a_2 \frac{\partial^2 y}{\partial s^2}(s,t)\vert_{s=s_k, t=b} ~ + ~ b_2\frac{\partial^2}{\partial s^2}\left[\frac{\partial y}{\partial t}(s,t)\right]\Bigl{|}_{s=s_k, t=b}.
\end{eqnarray*}
	Similar to the way we found $y_s(s_k)$ and $y_{ts}(s_k)$ in the shooting with Newton's method, we find $y_{ss}(s_k)$ and $z_{ss}(s_k)$ by solving the following system of FIVPs obtained by applying the operator $\frac{\partial^2}{\partial s^2}$ using HPCMs for each $k$:
	\begin{equation}\label{Halley-system}
		\begin{cases}
			D^{\alpha_1}_ay_{ss}(t) = w_{ss}(t), & y_{ss}(0) = 0,\\
			D^{1-\alpha_1}_aw_{ss}(t) = z_{ss}(t), & w_{ss}(0) = 0,\\
			D^{\alpha_2-1}_az_{ss}(t) = f_{ss}(t,y(t),w(t)), & z_{ss}(0) = 0.
		\end{cases}
	\end{equation}
	Since $t$ and $s$ are independent variables, $f_{ss}(t,y(t),w(t))$ can be written as
	\begin{equation*}
	f_y\cdot y_{ss} (t) + f_w \cdot w_{ss} (t) + f_{yy}\cdot y_s(t)^2 + f_{ww}\cdot w_s(t)^2 + f_{wy}\cdot w_s(t) y_s(t).
	\end{equation*}

	\subsection{High-order Predictor-Corrector Methods for system of FIVPs}\label{sec3-3}
	In order to find a $s_k$ with an acceptable accuracy, we iteratively solve systems of FIVPs \eqref{Newton-system} or \eqref{Halley-system}. Once we find the $s_k$, we solve systems of FIVPs \eqref{T_DIVP}, \eqref{T_RIVP}, \eqref{FDE_ODE-ver3}, or \eqref{FDE_ODE-ver3-RBC}.
	In this subsection, we describe how to deal with those systems of FIVPs using High-order Predictor-Corrector Methods (HPCMs) introduced in paper \cite{nguyen2017high}.
	Without loss of generality, we consider the following FIVP:
\begin{equation}\label{eq:apndx-fivp}
	\begin{cases}
		D^{\alpha}_0y(t) = f(t, y(t)), \ t\in[0,b], \\
		D^{(i)}y(0) = c_i, \ i=0, \ldots \lfloor \alpha\rfloor.
	\end{cases}	
\end{equation}
For convenience, let us denote $y_j$ as approximated value of $y(t_j)$ except for $y_0=c_0$ and let $f_j \equiv f(t_j,y_j)$, $y^c_j$ be a corrector of $y_j$, $y^p_j$ be a predictor of $y_j$, and $f_j^p\equiv f(t_j, y_j^p)$, $j=1,\cdots,N$. If $j=0$ then, $f_0 = f(0,c_0)$. We divide the domain $\Omega$ as follows:
\begin{equation*}
\Phi_N := \{\ t_j \,\vert\, a = t_0 < \cdots < t_j < \cdots < t_n < t_{n+1} < \cdots < t_N = b \}.
\end{equation*}

For simplicity, let step size be uniform which means $t_{j+1}-t_j = h, ~j = 0,1\cdots,N-1$. Then \eqref{eq:apndx-fivp} can be rewritten at time $t_{n+1}$ as follows:

\begin{equation*}
y(t_{n+1}) = g(t_{n+1}) + \frac{1}{\Gamma(\alpha)}\sum_{j=0}^{n}\int_{t_j}^{t_{j+1}}(t_{n+1}-\tau)^{\alpha-1}f(\tau,y(\tau))d\tau,
\end{equation*}

where $g(t_{n+1}) = \sum_{i=0}^{\lfloor\alpha\rfloor}\frac{(t_{n+1})^{i}}{i!}c_i.$ We interpolate $f(\tau,y(\tau))$ using linear or quadratic Lagrange polynomials over each interval $I_{j} = [ t_j,t_{j+1}],~j=0,1,\cdots,N-1$. Then we obtain the following predictor-corrector schemes:

\begin{enumerate}
	\item \emph{HPCM with linear Lagrange polynomial}:
	\begin{equation*}
	y^c_{n+1} = g(t_{n+1}) + \frac{1}{\Gamma(\alpha)}\left[\sum_{j=0}^{n-1}\left( B^{1,j}_{n+1}f_j+B^{2,j}_{n+1}f_{j+1} \right) + B^{1,n}_{n+1}f_n+B^{2,n}_{n+1}f^P_{n+1}\right],
	\end{equation*}
	where 
	\begin{alignat*}{2}
		B^{1,j}_{n+1} &= \frac{1}{h}\int_{t_j}^{t_{j+1}}(t_{n+1}-\tau)^{\alpha-1}(t_{j+1}-\tau)d\tau, \quad & B^{2,j}_{n+1} &= -\frac{1}{h}\int_{t_j}^{t_{j+1}}(t_{n+1}-\tau)^{\alpha-1}(t_j-\tau)d\tau, \\
		y^P_{n+1} &= g(t_{n+1}) + G_{\alpha, f}(t_{n+1}) + b^{1}_{n+1}f_{n-1} + b^{2}_{n+1}f_n,
		\quad & G_{\alpha, f}(t_{n+1}) &= \frac{1}{\Gamma(\alpha)}\sum_{j=0}^{n-1}\left( B^{1,j}_{n+1}f_j + B^{2,j}_{n+1}f_{j+1} \right), \\
		b^{1}_{n+1} &= \frac{1}{h}\int_{t_n}^{t_{n+1}}(t_{n+1}-\tau)^{\alpha-1}(t_n-\tau)d\tau,
		\quad & b^{2}_{n+1} &= -\frac{1}{h}\int_{t_n}^{t_{n+1}}(t_{n+1}-\tau)^{\alpha-1}(t_{n-1}-\tau)d\tau.
	\end{alignat*}
	\item \emph{HPCM with quadratic Lagrange polynomial}:
	\begin{equation*}
		\begin{aligned}
			y^c_{n+1} &= g(t_{n+1}) + \frac{1}{\Gamma(\alpha)}\bigg[ A^{1,0}_{n+1}f_0 + A^{2,0}_{n+1}f_{1/2} + A^{3,0}_{n+1}f_1 \\
			&\hspace{2cm}+\sum_{j=1}^{n-1}\left( A^{1,j}_{n+1}f_{j-1}+A^{2,j}_{n+1}f_j+A^{3,j}_{n+1}f_{j+1} \right) + A^{1,n}_{n+1}f_{n-1} + A^{2,n}_{n+1}f_n + A^{3,n}_{n+1}f^P_{n+1}\bigg],
		\end{aligned}
	\end{equation*}
	where
	\begin{alignat*}{2}
		A^{1,0}_{n+1} &= \frac{2}{h^2}\int_{t_j}^{t_{j+1}}(t_{n+1}-\tau)^{\alpha-1}(t_{1/2}-\tau)(t_1-\tau)d\tau, \quad & 
		A^{2,0}_{n+1} &= -\frac{4}{h^2}\int_{t_j}^{t_{j+1}}(t_{n+1}-\tau)^{\alpha-1}(t_0-\tau)(t_1-\tau)d\tau, \\
		A^{3,0}_{n+1} &= \frac{2}{h^2}\int_{t_j}^{t_{j+1}}(t_{n+1}-\tau)^{\alpha-1}(t_0-\tau)(t_{1/2}-\tau)d\tau, \quad & 
		A^{1,j}_{n+1} &= \frac{1}{2h^2}\int_{t_j}^{t_{j+1}}(t_{n+1}-\tau)^{\alpha-1}(t_j-\tau)(t_{j+1}-\tau)d\tau, \\
		A^{2,j}_{n+1} &= -\frac{1}{h^2}\int_{t_j}^{t_{j+1}}(t_{n+1}-\tau)^{\alpha-1}(t_{j-1}-\tau)(t_{j+1}-\tau)d\tau, \quad & 
		A^{3,j}_{n+1} &= \frac{1}{2h^2}\int_{t_j}^{t_{j+1}}(t_{n+1}-\tau)^{\alpha-1}(t_{j-1}-\tau)(t_j-\tau)d\tau,
	\end{alignat*}
	and the predictor $f^p_{n+1}$ is found as follows:
	\begin{equation*}
	y^P_{n+1} = g(t_{n+1}) + G_{\alpha, f}(t_{n+1}) + a^{1}_{n+1}f_{n-2} + a^{2}_{n+1}f_{n-1} + a^{3}_{n+1}f_n,
	\end{equation*}
	where
	\begin{equation*}
		\begin{aligned}
			G_{\alpha, f}(t_{n+1}) &= \frac{1}{\Gamma(\alpha)}\bigg[ A^{1,0}_{n+1}f_0 + A^{2,0}_{n+1}f_{1/2} + A^{3,0}_{n+1}y(t_1) + \sum_{j=0}^{n-1}\left( A^{1,j}_{n+1}f_{j-1} + A^{2,j}_{n+1}f_j + A^{3,j}_{n+1}f_{j+1} \right) \bigg],
		\end{aligned}
	\end{equation*}
	\begin{alignat*}{2}
		a^{1}_{n+1} &= \frac{1}{2h^2}\int_{t_n}^{t_{n+1}}(t_{n+1}-\tau)^{\alpha-1}(t_{n-1}-\tau)(t_n-\tau)d\tau, \quad &
		a^{2}_{n+1} &= -\frac{1}{h^2}\int_{t_n}^{t_{n+1}}(t_{n+1}-\tau)^{\alpha-1}(t_{n-2}-\tau)(t_n-\tau)d\tau, \\
		a^{3}_{n+1} &= \frac{1}{2h^2}\int_{t_n}^{t_{n+1}}(t_{n+1}-\tau)^{\alpha-1}(t_{n-2}-\tau)(t_{n-1}-\tau)d\tau. \quad & &
	\end{alignat*}

\end{enumerate}
We implement HPCMs in a FIVP. For the purpose of the implementation to the system \eqref{T_DIVP}, as an example, we find predictors $y^p_{n+1}, \ w^p_{n+1}, \ z^p_{n+1}$ with $s_0$ individually. We compute $f^p_{n+1}\equiv f(t_{n+1}, y^p_{n+1}, w^p_{n+1})$, and we find correctors $y^c_{n+1}, w^c_{n+1}, z^c_{n+1}$. Using proposed shooting techniques with HPCMs we find $s_1$ and then find predictors $y^p_{n+1}, \ w^p_{n+1}, \ z^p_{n+1}$ replacing $s_0$ with $s_1$, compute $f^p_{n+1}$, and update correctors $y^c_{n+1}, w^c_{n+1}, z^c_{n+1}$. We repeat this process until the absolute value of the approximated error function
\begin{equation}\label{eqn:app_err_ft}
|\tilde{F}(s_k)| = |a_2y^c_N(s_k) + b_2z^c_N(s_k) - \gamma_2|
\end{equation}
is within a tolerance. Similarly, we apply the scheme with HPCMs into other systems of FIVPs \eqref{T_RIVP}, \eqref{FDE_ODE-ver3}, and \eqref{FDE_ODE-ver3-RBC}.

\begin{remark}
	Alternatively, we only find predictors $y^p_{n+1}, \ w^p_{n+1}$. We then compute $z_{n+1}$ using $y^p_{n+1}, \ w^p_{n+1}$, and then we compute $w_{n+1}$ using $z_{n+1}$ as a predictor, and  $y_{n+1}$ using $w_{n+1}$ as a predictor. But it turns out that numerical results obtained by both ways of implementing HPCMs are nearly identical.
\end{remark}
The following theorems \cite{nguyen2017high} bound the Global error $E_{n+1}$ of the HPCM with linear and quadratic interpolations, respectively.
\begin{thm}{(Global Error of HPCM with Linear Interpolation)}\label{apndx2-thm-linear}
	Define $E_{n+1}$ to be global error. Suppose $f(\cdot,y(\cdot))\in C^2[a,b]$ and furthermore is Lipschitz continuous in the second argument, then we have
	\begin{equation*}
	E_{n+1} = \vert y(t_{n+1}) - \tilde{y}_{n+1} \vert \leq \mathcal{O}(h^2),
	\end{equation*}
	given $E_1\le Ch^2$.
\end{thm}
\begin{thm}{(Global Error of HPCM with Quadratic Interpolation)}\label{apndx2-thm-quad}
	Suppose $f(\cdot,y(\cdot))\in C^3[a,b]$ and is Lipshitz continuous in the second argument, then we have
	\begin{equation*}
	E_{n+1} \leq \mathcal{O}(h^3),
	\end{equation*}
	given $E_1,E_2 \leq \mathcal{O}(h^3)$ and $E_{1/2} \leq O(h^{3-\alpha}), \ 0 < \alpha < 1$.
\end{thm}

	\section{Numerical Examples}\label{sec4}
	In this section, we experimentally illustrate the performance of the proposed schemes. Numerically, we verify that our proposed schemes can deal with more complex FBVPs than the integral discretization schemes in \cite{Huang3, Huang2}. For that purpose, the proposed schemes are implemented in FNBVPs with $0 < \alpha_1 < 1$ whose exact solutions are polynomial, exponential, and sine functions in Examples \ref{ex.1} through \ref{ex.3}. We investigate absolute errors in maximum norm, convergence rates, and absolute values of the approximated error function $\vert \tilde{F}(s_k) \vert$ with various values of parameters. We discuss linear FBVPs with $\alpha_1=1$ whose exact solutions have low regularity and high regularity in Examples \ref{ex.4} and \ref{ex.5}, respectively. We compare numerical results obtained by our proposed schemes with the integral discretization schemes. But we emphasize that our proposed methods can deal with many different FBVPs unlike the another method in Examples \ref{ex.4} and \ref{ex.5}. Regarding the numerical results shown in the Appendix, let us summarize the parameters used
	\begin{itemize}
		\item $h$ denotes the size of time sub-interval.
		\item $s_0$ denotes the initial approximation of the sequence $\{s_k\}$ in proposed shooting methods.
		\item $k$ denotes the index of sequence $\{s_k\}$ generated by the proposed Newton's method or Halley's method. It can be considered as the number of iterations needed to meet a tolerance.
		\item $m$ denotes the maximum number of iterations in Newton's and Halley's methods.
		\item $Tol$ denotes the tolerance used to measure the error of the approximated error function $|\tilde{F}(s_k)|$ in Newton's method and Halley's method.
		\item $N$ denotes the number of time sub-intervals.
		\item Max. error denotes the pointwise absolute errors in the maximum norm. (i.e. $\max\limits_{1\le j \le N}\vert y^c_j - y(t_j) \vert$)
		\item $E_{\alpha, \beta}(t)$ denotes the two-parameter function of Mittag-Leffler type \cite{Podlubny}. 
	\end{itemize}
	In Examples \ref{ex.1} through \ref{ex.3}, we transform the FNBVP into the system of FIVPs \eqref{T_RIVP} and $s_0$ means an initial approximation to $y'(0)$. In Examples  \ref{ex.4} and \ref{ex.5}, the linear FBVP is transformed into the system of FIVPs \eqref{FDE_ODE-ver3-RBC} with $\epsilon = 10^{-10}$ and $s_0$ means an initial approximation to $y(0)$.
	For all Examples except for Example \ref{ex.4}, we implement the shooting technique based on Newton's method (Halley's method) combined with HPCM with linear (quadratic) interpolation to verify the order of convergence $\mathcal{O}(h^2)$ $\Bigl{(}\mathcal{O}(h^3)\Bigr{)}$, respectively.
	\begin{example}{Consider the following double-term FNBVP with RBCs}\label{ex.1}
		\begin{equation*}
			\begin{cases}
				D^{\alpha_2}_{0}y(t) = \frac{\Gamma(5)}{\Gamma(5-\alpha_2)}t^{4-\alpha_2}-\frac{\Gamma(5)}{\Gamma(5-\alpha_1)}t^{4-\alpha_1}-t^8+y^2+D_{0}^{\alpha_1}y(t),\\
				y(0) + y'(0)=0,~y(1) + y'(1)=5,
			\end{cases}
		\end{equation*}
		where the exact solution is $y(t)=t^4$.
	\end{example}

	

	\begin{example}{Consider the following double-term FNBVP with RBCs:}\label{ex.2}
		\begin{equation*}
			\begin{cases}
				D^{\alpha_2}_{0}y(t) = \lam^2t^{2-\alpha_2}E_{1,3-\alpha_2}(\lam t)-\left( \lam^2\frac{\Gamma(3)}{2\Gamma(3-\alpha_2)}t^{2-\alpha_2}+\lam^3\frac{\Gamma(4)}{6\Gamma(4-\alpha_2)}t^{3-\alpha_2} \right) - A^2 + y^2 -tB + tD^{\alpha_1}_{0}y(t),\\
				y(0)+y'(0) = 0,~y(1)+y'(1) \approx 0.2699,
			\end{cases}
		\end{equation*}
		where
		\begin{eqnarray*}
		\lambda &=& 1,\\
		A &=& e^{\lam t} - \left(1+\lam t + \frac{\lam^2}{2}t^2 +\frac{\lam^3}{3!}t^3 \right),\\
		B &=& \lam t^{1-\alpha_1}E_{1,2-\alpha_1}(\lam t) - \left( \lam\frac{\Gamma(2)}{\Gamma(2-\alpha_1)}t^{1-\alpha_1}+\lam^2\frac{\Gamma(3)}{2\Gamma(3-\alpha_1)}t^{2-\alpha_1}+\lam^3\frac{\Gamma(4)}{6\Gamma(4-\alpha_1)}t^{3-\alpha_1} \right),
	\end{eqnarray*}
		and the exact solution is
			$y(t) = e^{\lam t} - \left( 1+\lam t + \frac{\lam^2}{2}t^2 + \frac{\lam^3}{3!}t^3 \right)$.
	\end{example}

	\begin{example}{Consider the following double-term FNBVP with RBCs} \label{ex.3}
		\begin{equation*}
			\begin{cases}
				D^{\alpha_2}_{0}y(t) = F^{(\alpha_2)}_{\lam}(t)+\frac{\Gamma(4)}{6\Gamma(4-\alpha_2)}t^{3-\alpha_2}+\left(\sin(t)-t+\frac{t^3}{6}\right)^2-y^2-F^{(\alpha_1)}_{\lam}(t)+\frac{\Gamma(2)}{\Gamma(2-\alpha_1)}t^{1-\alpha_1}-\frac{\Gamma(4)}{6\Gamma(4-\alpha_1)}t^{3-\alpha_1}+D^{\alpha_1}_{0}y(t),\\
				y(0)+y'(0) = 0,~y(1)+y'(1) \approx 4.84399,
			\end{cases}
		\end{equation*}
		where
		\begin{eqnarray*}
		F_{\lambda}^{(\alpha)}(t) &=& -\frac{1}{2}i(i\lam)^{(\lceil \alpha \rceil)}t^{(\lceil \alpha \rceil-\alpha)}(E_{\lceil \alpha \rceil-\alpha+1}(i\lam t)-(-1)^{(\lceil \alpha \rceil)}E_{1,\lceil \alpha \rceil-\alpha+1}(-i\lam t)), \\
		\lambda &=& 1, \\
		\end{eqnarray*}
		and the exact solution is $y(t) = \sin(\lam t) - t + \frac{t^3}{6}$.
	\end{example}
	
	In Examples \ref{ex.1}, \ref{ex.2}, and \ref{ex.3}, we observe the following:
	\begin{enumerate}
		\item For all three examples, $w(t), \ z(t), \ f(t,y,w)$ (e.g. $D_0^{\alpha_1}y(t), \ y'(t), \ f(t,y,D_0^{\alpha_2}y(t))$) belong to $C^3[0,1]$. By Theorems \ref{apndx2-thm-linear} and \ref{apndx2-thm-quad}, thus, computed convergence profiles are estimated $\mathcal{O}(h^2)$ and $\mathcal{O}(h^3)$ for the HPCM with linear (quadratic) interpolation combined with shooting technique based on Newton's (Halley's) method, respectively.
		\item Tables \ref{TABLE-ex1-m}, \ref{TABLE-ex2-m}, \ref{TABLE-ex3-m} show the absolute values of approximated error function \eqref{eqn:app_err_ft} at $s_m$ (i.e. $|a_2y^c_N(s_m) + b_2z^c_N(s_m) - \gamma_2|$) versus the maximum number of iterations $m$ with various initial values $s_0$. $y^c_N(s_m)$ and $z^c_N(s_m)$ are computed by using proposed schemes. We set $\alpha_1 = 0.4, \alpha_2 = 1.7, h=0.01$ in all tables. From numerical results in those tables, we can verify that the sequence $\{s_k\}$ obtained by  proposed shooting algorithms approaches to the IC $s$ within the error at least $10^{-16}$ when $m$ is at most $10$ with $s_0=0.2, 0.4, 0.6, 0.8, 1.0$ for each. This leads us to the conclusion that proposed shooting techniques show a good performance with remarkable accuracy regarding to approximation of the IC $s$.
		\item Tables \ref{TABLE-ex1-maxerr}, \ref{TABLE-ex2-maxerr}, \ref{TABLE-ex3-maxerr} show pointwise absolute errors in the maximum norm and convergence rates computed versus the number of subintervals $N$ in the cases of $s_0=0.2, 1.0$ for each of  Newton's and Halley's method. We set $\alpha_1 = 0.4, \alpha_2 = 1.7$. The sequence $\{s_k\}$ was computed up to $s_{10}$ so, based on the observation of Tables \ref{TABLE-ex1-m}, \ref{TABLE-ex2-m}, \ref{TABLE-ex3-m}, we see that the error of the approximated error function $|\tilde{F}(s_k)|$ does not have an effect on the convergence rate of $y^c_j$ obtained by HPCMs. In Tables \ref{TABLE-ex1-maxerr}, \ref{TABLE-ex2-maxerr}, \ref{TABLE-ex3-maxerr}, we can see that computed convergence profiles obtained by proposed schemes approach 2 for Newton's method and 3 for Halley's method as $N$ is increased. Thus numerical results shown in those tables support that proposed methods follow global error estimates of HPCMs.
		\item Proposed methods are tested for a variety of values of $\alpha_1, \alpha_2$ and numerical results for each pair of $(\alpha_1, \alpha_2)$ are shown in Tables \ref{TABLE-ex1-alpha}, \ref{TABLE-ex2-alpha}, \ref{TABLE-ex3-alpha}. For each pair of fractional orders $(\alpha_1, \alpha_2)$ pointwise absolute errors in the maximum norm, computed convergence rates, and number of iterations $k$ such that $|\tilde{F}(s_k)| < Tol$ versus the number of subintervals $N$ are listed in the tables. The initial approximation to $s$ was set $s_0 = 0.2$ in all three tables. In order to minimize the number of iterations $k$, the tolerance was set $Tol=10^{-5}$ for Newton's method and $Tol=10^{-10}$ for Halley's method in Table \ref{TABLE-ex1-alpha}, $Tol=10^{-10}$ for both shooting techniques in Table \ref{TABLE-ex2-alpha}, and $Tol=10^{-15}$ for Newton's method and $Tol=10^{-16}$ for Halley's method in Table \ref{TABLE-ex3-alpha}. Numerical results shown in the tables demonstrate that for all suggested pairs of fractional orders, rates of convergence approach 2 for Newton's method, 3 for Halley's method that are theoretical convergence rates of HPCMs. In Examples \ref{TABLE-ex1-alpha} and \ref{TABLE-ex3-alpha}, we observe that the number of iterations $k$ required to meet the tolerance at $(0.9, 1.1)$ is relatively greater than other pairs of fractional orders for both Newton's and Halley's method.
	\end{enumerate}

	
	\begin{example}{Consider the following single-term Linear FBVP with RBCs} \label{ex.4}
		\begin{equation*}
			\begin{cases}
				D^{\alpha_2}_{0}y(t) = \varphi(t)-(2t+6)y'(t)\\
				y(0)-\frac{1}{1-\alpha_2}y'(0) = \gamma_1,~y(1)+y'(1) = \gamma_2,
			\end{cases}
		\end{equation*}
		where $1 < \alpha_2 < 2$,
		\begin{equation*}
		\begin{split}
	\varphi(t) = \frac{\Gamma(\alpha_2+1)}{\Gamma(1)}+\frac{\Gamma(2\alpha_2)}{\Gamma(\alpha_2)}t^{(\alpha_2-1)}+4\frac{\Gamma(4)}{\Gamma(4-\alpha_2)}t^{(3-\alpha_2)}+&\frac{\Gamma(5)}{\Gamma(5-\alpha_2)}t^{(4-\alpha_2)}\\ 
		+&(2t+6)(\alpha_2 t^{\alpha_2-1}+(2\alpha_2-1)t^{2\alpha_2-2}+3+12t^2+4t^3),
		\end{split}
		\end{equation*}
		and the exact solution is $y(t) = t^{\alpha_2}+t^{2\alpha_2-1}+1+3t+4t^3+t^4$ \cite{Huang2}.
	\end{example}
	Since $D_0^{\alpha_1}y(t), \ y'(t), \ D_0^{\alpha_2}y(t)$ do not belong to $C^3[0,1]$, global error estimates of HPCMs in Theorems \ref{apndx2-thm-linear} and \ref{apndx2-thm-quad} cannot be applied to Example \ref{ex.4}. Alternatively, we adopt the linear explicit method described in \ref{apndx3} with proposed shooting techniques. 
	In this example, we compare the accuracy and convergence rate of the approximate solution obtained by the proposed shooting technique based on Newton's method with the modified integral discretization scheme \cite{Huang2} for each $\alpha_2 = 1.1, 1.3, 1.5, 1.7, 1.9$. In Table \ref{TABLE-ex4}, we can observe that our proposed method shows equal performance to the modified integral discretization scheme \cite{Huang2}.
	
	
	\begin{example}{Consider the following single-term Linear FBVP with RBCs} \label{ex.5}
		\begin{equation*}
			\begin{cases}
				D^{\alpha_2}_{0}y(t) = F(t)-\cos(t)y(t)-\sin(t)y'(t)\\
				y(0)-\frac{1}{1-\alpha_2}y'(0) = \gamma_1,~y(1)+y'(1) = \gamma_2
			\end{cases}
		\end{equation*}
		whose $1 < \alpha < 2$ and the exact solution is
		\begin{equation*}
			y(t) = \sin(\lam t) - t + \frac{t^3}{6}.
		\end{equation*}
	\end{example}
	
	In Example \ref{ex.5}, we compare the performance of our proposed methods with the modified integral discretization scheme \cite{Huang2}. Table \ref{TABLE-ex5} shows pointwise absolute errors and computed convergence profiles versus the number of subintervals for each $\alpha_2 = 1.1, 1.3, 1.5, 1.7, 1.9$. In Table \ref{TABLE-ex5}, we can see that for all values of $\alpha_2$, the computed rates of convergence obtained by the proposed shooting technique based on Halley's method combined with third order HPCM are around 3.0 while the computed rates of convergence obtained by the modified integral discretization scheme \cite{Huang2} are around 2.0. \\
	The algorithm of the proposed shooting techniques with second-order HPCM requires less than the number of arithmetic operations needed by the modified integral discretization scheme to solve a FBVP with RBCs than the modified integral discretization scheme \cite{Huang2} because the predictor and corrector in HPCMs share the computation of the memory effect. In practice this results the proposed shooting technique based on Newton's method consume less CPU than the modified discretization scheme \cite{Huang2} and the CPU time executed by the proposed shooting technique based on Halley's method is approximately equal to the CPU time executed by the modified integral discretization scheme \cite{Huang2} as shown in Table \ref{TABLE-ex5}.

	
	\section{Conclusion}\label{sec5}
	We introduced new numerical schemes for solving FNBVPs with any RBCs. The idea was to transform a FNBVP into a system of FIVPs. By doing that we could adopt a pre-existing numerical method for solving the system of FIVPs and we mainly employed HPCMs. The unknown IC $s$ in the system was approximated by proposed shooting methods based on Newton's and Halley's method and this is the main algorithm of proposed schemes.
	Under the assumption that $m$ is large enough so that $|\tilde{F}(s_m)|$ is small enough, theoretical convergence rates of proposed methods were $\mathcal{O}(h^2)$ for shooting with Newton's method and $\mathcal{O}(h^3)$ for shooting with Halley's method on account of global error estimates in HPCMs.\\
	In Examples \ref{ex.1} through \ref{ex.3}, we verified that proposed schemes can handle double-term FNVBPs with RBCs whose exact solutions include polynomial, exponential, and sine function. Convergence profiles obtained by proposed schemes were computed as expected by the global error estimates. However, Tables \ref{TABLE-ex1-alpha}, \ref{TABLE-ex2-alpha}, and \ref{TABLE-ex3-alpha} point out that the convergence rate of the sequence $\{s_k\}$ depends on fractional orders. We still need to address an error analysis of shooting techniques based on Newton's and Halley's methods for solving a system of FIVPs. This will be considered in a subsequent paper.
	Examples \ref{ex.4} and \ref{ex.5} demonstrated the performance of proposed methods for solving single-term linear FBVPs with exact solutions having low regularity and high regularity, respectively. Tables \ref{TABLE-ex4} and \ref{TABLE-ex5} showed that proposed methods can deal with not only nonlinear FBVPs but also linear FBVPs. In Example \ref{ex.4}, we adopted the linear explicit method described in \ref{apndx3} and this shows that the proposed shooting techniques can be assembled with not only HPCMs but also other pre-existing numerical schemes for solving a system of FIVPs.
	In Example \ref{ex.5}, we observed that computed convergence rates obtained by our proposed shooting technique based on Halley's method with third order HPCM are higher than the modified integral discretization scheme \cite{Huang2}.

	\section*{Acknowledgments}
	This work was supported  by the National Research Foundation of Korea(NRF)
	grant funded by the Korea government(MSIP) (NRF-2016R1D1A1B03935514). 

\appendix
\section{Tables of Numerical Results}\label{apndx1}
\subsection{Example \ref{ex.1}: FNBVP whose type of exact solution is polynomial}
	\begin{table}[H]
	\centering
	\begin{adjustbox}{max width=\textwidth}
		\begin{tabular}{c|ccccc|ccccc}
			\toprule
			&  \multicolumn{5}{c}{Newton's method} & \multicolumn{5}{|c}{Halley's method}\\
			\midrule
			$m$ & $s_0=0.2$ & $s_0=0.4$ & $s_0=0.6$ & $s_0=0.8$ & $s_0=1.0$ & $s_0=0.2$ & $s_0=0.4$ & $s_0=0.6$ & $s_0=0.8$ & $s_0=1.0$ \\
			\toprule
			1     & 0.604011 & 1.305541 & 2.105771 & 3.00776  & 4.017051 & 0.603191 & 1.304704 & 2.104838 & 3.006644 & 4.01565 \\
			2     & 0.042234 & 1.48E-01 & 0.303846 & 5.05E-01 & 0.753495 & 0.054329 & 1.90E-01 & 0.392297 & 6.59E-01 & 0.995211 \\
			3     & 0.000274 & 3.25E-03 & 0.01249  & 3.10E-02 & 0.061341 & 0.000604 & 6.84E-03 & 0.025715 & 6.31E-02 & 0.124161 \\
			4     & 7.31E-08 & 7.33E-07 & 2.17E-05 & 1.46E-04 & 0.000578 & 9.11E-08 & 8.24E-06 & 1.34E-04 & 0.000813 & 0.003036 \\
			5     & 2.28E-11 & 2.29E-10 & 6.71E-09 & 4.22E-08 & 1.25E-07 & 2.58E-11 & 2.32E-09 & 3.41E-08 & 8.56E-08 & 1.15E-06 \\
			6     & 7.11E-15 & 7.19E-14 & 2.10E-12 & 1.32E-11 & 3.91E-11 & 6.22E-15 & 6.55E-13 & 9.64E-12 & 2.42E-11 & 3.26E-10 \\
			7     & 0        & 0        & 8.88E-16 & 4.44E-15 & 1.07E-14 & 8.88E-16 & 0        & 2.66E-15 & 6.22E-15 & 9.15E-14 \\
			8     & 0        & 0        & 0        & 0        & 1.78E-15 & 8.88E-16 & 0        & 8.88E-16 & 0        & 8.88E-16 \\
			9     & 0        & 0        & 0        & 0        & 0        & 8.88E-16 & 0        & 8.88E-16 & 0        & 8.88E-16 \\
			10    & 0        & 0        & 0        & 0        & 0        & 8.88E-16 & 0        & 8.88E-16 & 0        & 8.88E-16 \\
			\bottomrule
		\end{tabular}%
	\end{adjustbox}
	\caption{Errors of shooting techniques $\vert \tilde{F}(s_m) \vert$ in \eqref{eqn:app_err_ft}  versus the maximum number of iterations with various values of $s_0$ in Example \ref{ex.1}. We set $h=0.01, \alpha_1=0.4, \alpha_2=1.7$}
	\label{TABLE-ex1-m}
\end{table}%

\begin{table}[H]
	\centering
	\begin{adjustbox}{max width=\textwidth}
		\begin{tabular}{l|cccc|cccc}
			\toprule
			& \multicolumn{4}{c}{Newton's method}     &      \multicolumn{4}{|c}{Halley's method} \\
			\midrule
			& \multicolumn{2}{c}{$s_0=0.2$} & \multicolumn{2}{c}{$s_0=1.0$} & \multicolumn{2}{|c}{$s_0=0.2$} & \multicolumn{2}{c}{$s_0=1.0$} \\
			\midrule
			$N$ & Max. error & Rate & Max. error & Rate & Max. error & Rate & Max. error & Rate \\ 
			\toprule
			10    & 1.17E-02 & -     & 1.17E-02 & -     & 1.02E-03 & -     & 1.02E-03 & -         \\
			20    & 5.63E-03 & 1.059 & 5.63E-03 & 1.059 & 1.83E-04 & 2.486 & 1.83E-04 & 2.486 \\
			40    & 1.78E-03 & 1.665 & 1.78E-03 & 1.665 & 3.63E-05 & 2.333 & 3.63E-05 & 2.332 \\
			80    & 4.92E-04 & 1.855 & 4.92E-04 & 1.854 & 5.47E-06 & 2.730 & 5.47E-06 & 2.730 \\
			160   & 1.29E-04 & 1.931 & 1.29E-04 & 1.931 & 7.51E-07 & 2.866 & 7.51E-07 & 2.866 \\
			320   & 3.30E-05 & 1.966 & 3.30E-05 & 1.966 & 9.88E-08 & 2.925 & 9.88E-08 & 2.925 \\
			\bottomrule
		\end{tabular}%
	\end{adjustbox}
	\caption{Pointwise absolute errors in the maximum norm and computed rates of convergence versus the number of subintervals $N$ with $s_0 = 0.2, 1.0$ for each in Example \ref{ex.1}. We set $\alpha_1 = 0.4, \alpha_2 = 1.7$, and $m$ is fixed at 10.}
	\label{TABLE-ex1-maxerr}
\end{table}%

\begin{table}[H]
	\centering
	\begin{adjustbox}{max width=\textwidth}
		\begin{tabular}{ll|ccccc|ccccc}
			\toprule
			& & \multicolumn{5}{c}{Newton's method}     &      \multicolumn{5}{|c}{Halley's method} \\
			\midrule
			& & \multicolumn{5}{c}{$(\alpha_1, \alpha_2)$}     &      \multicolumn{5}{|c}{$(\alpha_1, \alpha_2)$} \\
			\midrule
			$N$ & & $(0.9, 1.1)$ & $(0.7, 1.3)$ & $(0.5,1.5)$ & $(0.3, 1.7)$ & $(0.1, 1.9)$ & $(0.9, 1.1)$ & $(0.7, 1.3)$ & $(0.5,1.5)$ & $(0.3, 1.7)$ & $(0.1, 1.9)$\\
			\midrule
			\toprule
			64    & Max. error & 1.17E-03 & 9.45E-05 & 6.06E-04 & 7.72E-04 & 7.94E-04 & 5.05E-05 & 4.17E-06 & 9.12E-06 & 1.03E-05 & 7.17E-06 \\
			& Rate  &   -    &     -  &  -     &     -  &     -  &  -     &   -    &   -    &   -    & - \\
			& k     & 8     & 5     & 4     & 4     & 4     & 13    & 7     & 6     & 5     & 5 \\
			128   & Max. error & 2.55E-04 & 6.60E-05 & 1.72E-04 & 2.05E-04 & 2.09E-04 & 5.61E-06 & 2.42E-07 & 1.38E-06 & 1.46E-06 & 1.05E-06 \\
			& Rate  & 2.196 & 5.176 & 1.817 & 1.911 & 1.925 & 3.168 & 4.105 & 2.728 & 2.826 & 2.777 \\
			& k     & 8     & 4     & 4     & 4     & 4     & 11    & 7     & 6     & 5     & 5 \\
			256   & Max. error & 5.46E-05 & 2.44E-05 & 4.59E-05 & 5.29E-05 & 5.38E-05 & 6.24E-07 & 7.59E-08 & 1.89E-07 & 1.95E-07 & 1.43E-07 \\
			& Rate  & 2.223 & 1.435 & 1.907 & 1.955 & 1.959 & 3.170 & 1.673 & 2.866 & 2.905 & 2.871 \\
			& k     & 7     & 4     & 4     & 4     & 4     & 12    & 6     & 5     & 5     & 5 \\
			512   & Max. error & 1.16E-05 & 7.58E-06 & 1.19E-05 & 1.35E-05 & 1.37E-05 & 6.85E-08 & 1.39E-08 & 2.48E-08 & 2.53E-08 & 1.90E-08 \\
			& Rate  & 2.236 & 1.687 & 1.949 & 1.976 & 1.976 & 3.186 & 2.449 & 2.929 & 2.942 & 2.914 \\
			& k     & 8     & 4     & 4     & 4     & 4     & 12    & 6     & 5     & 5     & 5 \\
			1024  & Max. error & 2.41E-06 & 2.15E-06 & 3.03E-06 & 3.39E-06 & 3.45E-06 & 7.45E-09 & 2.12E-09 & 3.19E-09 & 3.25E-09 & 2.48E-09 \\
			& Rate  & 2.261 & 1.818 & 1.970 & 1.988 & 1.987 & 3.202 & 2.712 & 2.960 & 2.962 & 2.936 \\
			& k     & 8     & 4     & 4     & 4     & 4     & 11    & 6     & 5     & 5     & 5 \\
			2048  & Max. error & 4.95E-07 & 5.31E-07 & 7.67E-07 & 8.49E-07 & 8.63E-07 & 8.03E-10 & 3.02E-10 & 4.01E-10 & 4.18E-10 & 3.21E-10 \\
			& Rate  & 2.288 & 2.018 & 1.984 & 1.998 & 1.999 & 3.214 & 2.811 & 2.988 & 2.961 & 2.949 \\
			& k     & 8     & 4     & 4     & 4     & 4     & 11    & 6     & 5     & 5     & 5 \\
			\bottomrule
		\end{tabular}%
	\end{adjustbox}
	\caption{Pointwise absolute errors in the maximum norm and computed rates of convergence obtained by the proposed schemes with various fractional orders versus the number of subintervals in Example \ref{ex.1}. We set $s_0=0.2, Tol=10^{-5}$ for shooting with Newton's method and $s_0=0.2, Tol=10^{-10}$ for Halley's method.}
	\label{TABLE-ex1-alpha}
\end{table}%

\subsection{Example \ref{ex.2}: FNBVP whose exact solution involves an exponential function}
	\begin{table}[H]
	\centering
	\begin{adjustbox}{max width=\textwidth}
		\begin{tabular}{c|ccccc|ccccc}
			\toprule
			&  \multicolumn{5}{c}{Newton's method} & \multicolumn{5}{|c}{Halley's method}\\
			\midrule
			$m$ & $s_0=0.2$ & $s_0=0.4$ & $s_0=0.6$ & $s_0=0.8$ & $s_0=1.0$ & $s_0=0.2$ & $s_0=0.4$ & $s_0=0.6$ & $s_0=0.8$ & $s_0=1.0$ \\
			\toprule
			1     & 0.414113 & 0.890547 & 1.42639  & 2.02066  & 2.674035 & 0.414046 & 0.890413 & 1.426145 & 2.020256 & 2.673415 \\
			2     & 0.027748 & 9.51E-02 & 0.190648 & 3.11E-01 & 0.456212 & 0.03142  & 1.08E-01 & 0.219523 & 3.66E-01 & 0.549419 \\
			3     & 0.000177 & 1.93E-03 & 0.007068 & 1.70E-02 & 0.032652 & 0.000258 & 2.80E-03 & 0.010391 & 2.54E-02 & 0.050225 \\
			4     & 1.51E-08 & 9.56E-07 & 1.20E-05 & 6.71E-05 & 0.000243 & 1.61E-08 & 2.10E-06 & 2.89E-05 & 1.70E-04 & 0.000646 \\
			5     & 6.59E-13 & 4.20E-11 & 5.57E-10 & 3.99E-09 & 2.45E-08 & 1.25E-13 & 1.52E-11 & 1.33E-12 & 6.56E-09 & 1.08E-07 \\
			6     & 0        & 1.89E-15 & 2.44E-14 & 1.75E-13 & 1.07E-12 & 5.55E-17 & 1.67E-16 & 1.11E-16 & 5.11E-14 & 8.40E-13 \\
			7     & 0        & 0        & 0        & 0        & 1.67E-16 & 5.55E-17 & 5.55E-17 & 5.55E-17 & 5.55E-17 & 5.55E-17 \\
			8     & 0        & 0        & 0        & 0        & 5.55E-17 & 5.55E-17 & 5.55E-17 & 5.55E-17 & 5.55E-17 & 0 \\
			9     & 0        & 0        & 0        & 0        & 0        & 5.55E-17 & 5.55E-17 & 5.55E-17 & 5.55E-17 & 0 \\
			10    & 0        & 0        & 0        & 0        & 0        & 5.55E-17 & 5.55E-17 & 5.55E-17 & 5.55E-17 & 0 \\
			\bottomrule
		\end{tabular}%
	\end{adjustbox}
	\caption{Errors of shooting techniques $\vert \tilde{F}(s_m) \vert$ in \eqref{eqn:app_err_ft} versus the maximum number of iterations $m$ for each $s_0$ in Example \ref{ex.2} as we set $\alpha_1 = 0.4, \alpha_2=1.7, h=0.01$.}
	\label{TABLE-ex2-m}
\end{table}%

\begin{table}[H]
	\centering
	\begin{adjustbox}{max width=\textwidth}
		\begin{tabular}{l|cccc|cccc}
			\toprule
			& \multicolumn{4}{c}{Newton's method}     &      \multicolumn{4}{|c}{Halley's method} \\
			\midrule
			& \multicolumn{2}{c}{$s_0=0.2$} & \multicolumn{2}{c}{$s_0=1.0$} & \multicolumn{2}{|c}{$s_0=0.2$} & \multicolumn{2}{c}{$s_0=1.0$} \\
			\midrule
			$N$ & Max. error & Rate & Max. error & Rate & Max. error & Rate & Max. error & Rate \\ 
			\toprule
			10    & 1.10E-03 & -     & 1.10E-03 & -     & 5.65E-05 & -     & 5.65E-05 & -         \\
			20    & 4.05E-04 & 1.434 & 4.05E-04 & 1.434 & 1.71E-05 & 1.725 & 1.71E-05 & 1.725 \\
			40    & 1.19E-04 & 1.765 & 1.19E-04 & 1.765 & 2.82E-06 & 2.599 & 2.82E-06 & 2.599 \\
			80    & 3.21E-05 & 1.893 & 3.21E-05 & 1.893 & 3.99E-07 & 2.824 & 3.99E-07 & 2.824 \\
			160   & 8.32E-06 & 1.949 & 8.32E-06 & 1.949 & 5.29E-08 & 2.913 & 5.29E-08 & 2.913 \\
			320   & 2.12E-06 & 1.974 & 2.12E-06 & 1.974 & 6.84E-09 & 2.952 & 6.84E-09 & 2.952 \\
			\bottomrule
		\end{tabular}%
	\end{adjustbox}
	\caption{Pointwise absolute errors in the maximum norm and computed rates of convergence versus the number of subintervals $N$ in Example \ref{ex.2} as we set $\alpha_1 = 0.4, \alpha_2 = 1.7, h=0.01,m=10$ and no $Tol$. }
	\label{TABLE-ex2-maxerr}
\end{table}%

\begin{table}[H]
	\centering
	\begin{adjustbox}{max width=\textwidth}
		\begin{tabular}{ll|ccccc|ccccc}
			\toprule
			& & \multicolumn{5}{c}{Newton's method}     &      \multicolumn{5}{|c}{Halley's method} \\
			\midrule
			& & \multicolumn{5}{c}{$(\alpha_1, \alpha_2)$}     &      \multicolumn{5}{|c}{$(\alpha_1, \alpha_2)$} \\
			\midrule
			$N$ & & $(0.9, 1.1)$ & $(0.7, 1.3)$ & $(0.5,1.5)$ & $(0.3, 1.7)$ & $(0.1, 1.9)$ & $(0.9, 1.1)$ & $(0.7, 1.3)$ & $(0.5,1.5)$ & $(0.3, 1.7)$ & $(0.1, 1.9)$\\
			\midrule
			\toprule
			64    & Max. error & 1.25E-04 & 1.10E-05 & 4.15E-05 & 4.97E-05 & 5.10E-05 & 3.63E-06 & 9.39E-08 & 6.56E-07 & 7.63E-07 & 6.59E-07 \\
			& Rate  &  -     &    -   &    -   &  -     &    -   &   -    &    -   &   -    &    -   &  -\\
			& k     & 5     & 5     & 5     & 5     & 5     & 6     & 5     & 5     & 5     & 5 \\
			128   & Max. error & 2.71E-05 & 5.55E-06 & 1.14E-05 & 1.29E-05 & 1.32E-05 & 4.00E-07 & 3.39E-08 & 9.36E-08 & 1.02E-07 & 8.86E-08 \\
			& Rate  & 2.201 & 9.891 & 1.861 & 1.942 & 1.955 & 3.181 & 1.471 & 2.808 & 2.899 & 2.896 \\
			& k     & 5     & 5     & 5     & 5     & 5     & 6     & 5     & 5     & 5     & 5 \\
			256   & Max. error & 5.80E-06 & 1.89E-06 & 3.00E-06 & 3.30E-06 & 3.35E-06 & 4.36E-08 & 6.32E-09 & 1.25E-08 & 1.33E-08 & 1.16E-08 \\
			& Rate  & 2.226 & 1.555 & 1.928 & 1.970 & 1.974 & 3.197 & 2.421 & 2.905 & 2.944 & 2.935 \\
			& k     & 5     & 5     & 5     & 5     & 5     & 6     & 5     & 5     & 5     & 5 \\
			512   & Max. error & 1.22E-06 & 5.61E-07 & 7.71E-07 & 8.34E-07 & 8.47E-07 & 4.72E-09 & 1.04E-09 & 1.62E-09 & 1.70E-09 & 1.49E-09 \\
			& Rate  & 2.249 & 1.752 & 1.962 & 1.984 & 1.983 & 3.209 & 2.604 & 2.950 & 2.966 & 2.955 \\
			& k     & 5     & 5     & 5     & 5     & 5     & 6     & 5     & 5     & 5     & 5 \\
			1024  & Max. error & 2.52E-07 & 1.56E-07 & 1.96E-07 & 2.10E-07 & 2.14E-07 & 5.06E-10 & 1.52E-10 & 2.06E-10 & 2.16E-10 & 1.91E-10 \\
			& Rate  & 2.273 & 1.848 & 1.979 & 1.991 & 1.988 & 3.221 & 2.779 & 2.973 & 2.978 & 2.965 \\
			& k     & 5     & 5     & 5     & 5     & 5     & 6     & 5     & 5     & 5     & 5 \\
			2048  & Max. error & 5.12E-08 & 4.17E-08 & 4.93E-08 & 5.27E-08 & 5.37E-08 & 5.39E-11 & 2.12E-11 & 2.58E-11 & 2.75E-11 & 2.44E-11 \\
			& Rate  & 2.301 & 1.902 & 1.988 & 1.994 & 1.991 & 3.232 & 2.839 & 2.999 & 2.973 & 2.972 \\
			& k     & 5     & 5     & 5     & 5     & 5     & 6     & 5     & 5     & 5     & 5 \\
			\bottomrule
		\end{tabular}%
	\end{adjustbox}
	\caption{Pointwise absolute errors in the maximum norm and computed rates of convergence obtained by the proposed schemes with various fractional orders versus the number of subintervals $N$ in Example \ref{ex.2} as we set $s_0=0.2, Tol=10^{-10}$ for both shooting techniques}
	\label{TABLE-ex2-alpha}
\end{table}%

\subsection{Example \ref{ex.3}: FNBVP whose exact solution involves a sine function}
	\begin{table}[H]
	\centering
	\begin{adjustbox}{max width=\textwidth}
		\begin{tabular}{c|ccccc|ccccc}
			\toprule
			&  \multicolumn{5}{c}{Newton's method} & \multicolumn{5}{|c}{Halley's method}\\
			\midrule
			$m$ & $s_0=0.2$ & $s_0=0.4$ & $s_0=0.6$ & $s_0=0.8$ & $s_0=1.0$ & $s_0=0.2$ & $s_0=0.4$ & $s_0=0.6$ & $s_0=0.8$ & $s_0=1.0$ \\
			\toprule
			1     & 0.487617 & 0.891569 & 1.202993 & 1.407814 & 1.485109 & 0.487722 & 0.891834 & 1.203474 & 1.408578 & 1.486244 \\
			2     & 0.101684 & 3.44E-01 & 0.653431 & 9.74E-01 & 1.258505 & 0.106265 & 3.56E-01 & 0.668984 & 9.90E-01 & 1.270771 \\
			3     & 0.004491 & 5.07E-02 & 0.182979 & 4.14E-01 & 0.72559  & 0.005185 & 5.68E-02 & 0.199315 & 4.40E-01 & 0.757881 \\
			4     & 1.26E-05 & 1.14E-03 & 0.014409 & 0.073158 & 0.226095 & 1.69E-05 & 1.51E-03 & 1.80E-02 & 8.68E-02 & 0.255888 \\
			5     & 1.13E-08 & 1.57E-06 & 0.000101 & 0.002343 & 0.021949 & 1.54E-08 & 2.39E-06 & 1.63E-04 & 3.48E-03 & 2.96E-02 \\
			6     & 1.01E-11 & 1.41E-09 & 9.49E-08 & 4.43E-06 & 0.000225 & 1.39E-11 & 2.16E-09 & 1.59E-07 & 8.62E-06 & 4.22E-04 \\
			7     & 9.06E-15 & 1.26E-12 & 8.48E-11 & 3.97E-09 & 2.22E-07 & 1.25E-14 & 1.95E-12 & 1.44E-10 & 7.82E-09 & 4.62E-07 \\
			8     & 6.94E-18 & 1.12E-15 & 7.58E-14 & 3.55E-12 & 1.99E-10 & 6.94E-18 & 1.76E-15 & 1.30E-13 & 7.06E-12 & 4.17E-10 \\
			9     & 0        & 1.39E-17 & 6.94E-17 & 3.16E-15 & 1.78E-13 & 6.94E-18 & 6.94E-18 & 1.25E-16 & 6.37E-15 & 3.77E-13 \\
			10    & 0        & 6.94E-18 & 6.94E-18 & 6.94E-18 & 1.67E-16 & 6.94E-18 & 0        & 6.94E-18 & 6.94E-18 & 3.33E-16 \\
			\bottomrule
		\end{tabular}%
	\end{adjustbox}
	\caption{Errors of shooting techniques $\vert \tilde{F}(s_m) \vert$ in \eqref{eqn:app_err_ft}  versus the maximum number of iterations $m$ with various initial approximations $s_0$ in Example \ref{ex.3} as we set $\alpha_1 = 0.4, \alpha_2 = 1.7, h = 0.01$.}
	\label{TABLE-ex3-m}
\end{table}%

\begin{table}[H]
	\centering
	\begin{adjustbox}{max width=\textwidth}
		\begin{tabular}{l|cccc|cccc}
			\toprule
			& \multicolumn{4}{c}{Newton's method}     &      \multicolumn{4}{|c}{Halley's method} \\
			\midrule
			& \multicolumn{2}{c}{$s_0=0.2$} & \multicolumn{2}{c}{$s_0=1.0$} & \multicolumn{2}{|c}{$s_0=0.2$} & \multicolumn{2}{c}{$s_0=1.0$} \\
			\midrule
			$N$ & Max. error & Rate & Max. error & Rate & Max. error & Rate & Max. error & Rate \\ 
			\toprule
			10    & 1.73E-04 & -     & 1.73E-04 & -     & 1.19E-05 & -     & 1.19E-05 & -     \\
			20    & 6.85E-05 & 1.338 & 6.85E-05 & 1.338 & 3.97E-06 & 1.578 & 3.97E-06 & 1.578 \\
			40    & 2.07E-05 & 1.727 & 2.07E-05 & 1.727 & 6.62E-07 & 2.586 & 6.62E-07 & 2.586 \\
			80    & 5.63E-06 & 1.877 & 5.63E-06 & 1.877 & 9.28E-08 & 2.834 & 9.28E-08 & 2.834 \\
			160   & 1.47E-06 & 1.942 & 1.47E-06 & 1.942 & 1.22E-08 & 2.926 & 1.22E-08 & 2.926 \\
			320   & 3.74E-07 & 1.971 & 3.74E-07 & 1.971 & 1.56E-09 & 2.966 & 1.56E-09 & 2.966 \\
			\bottomrule
		\end{tabular}%
	\end{adjustbox}
	\caption{Pointwise absolute errors in the maximum norm and computed rates of convergence versus the number of subintervals $N$ in Example \ref{ex.3} as we set $\alpha_1 = 0.4, \alpha_2 = 1.7, m=10$.}
	\label{TABLE-ex3-maxerr}
\end{table}%

\begin{table}[H]
	\centering
	\begin{adjustbox}{max width=\textwidth}
		\begin{tabular}{ll|ccccc|ccccc}
			\toprule
			& & \multicolumn{5}{c}{Newton's method}     &      \multicolumn{5}{c}{Halley's method} \\
			\midrule
			& & \multicolumn{5}{c}{$(\alpha_1, \alpha_2)$}     &      \multicolumn{5}{|c}{$(\alpha_1, \alpha_2)$} \\
			\midrule
			$N$ & & $(0.9, 1.1)$ & $(0.7, 1.3)$ & $(0.5,1.5)$ & $(0.3, 1.7)$ & $(0.1, 1.9)$ & $(0.9, 1.1)$ & $(0.7, 1.3)$ & $(0.5,1.5)$ & $(0.3, 1.7)$ & $(0.1, 1.9)$\\
			\midrule
			\toprule
			64    & Max. error & 2.26E-05 & 6.92E-07 & 6.70E-06 & 8.99E-06 & 9.95E-06 & 7.65E-07 & 5.55E-08 & 1.33E-07 & 1.84E-07 & 1.95E-07 \\
			& Rate  & -      & -      & -      & -      & -      & -      & -      & -      & -      & - \\
			& k     & 20    & 10    & 9     & 8     & 6     & 22    & 10    & 9     & 8     & 7 \\
			128   & Max. error & 5.04E-06 & 6.51E-07 & 1.88E-06 & 2.35E-06 & 2.57E-06 & 8.54E-08 & 2.93E-09 & 1.95E-08 & 2.43E-08 & 2.55E-08 \\
			& Rate  & 2.166 & 8.900 & 1.833 & 1.937 & 1.954 & 3.164 & 4.244 & 2.767 & 2.920 & 2.933 \\
			& k     & 20    & 10    & 9     & 8     & 7     & 22    & 11    & 9     & 8     & 7 \\
			256   & Max. error & 1.10E-06 & 2.50E-07 & 4.99E-07 & 6.00E-07 & 6.54E-07 & 9.41E-09 & 8.67E-10 & 2.64E-09 & 3.11E-09 & 3.27E-09 \\
			& Rate  & 2.192 & 1.381 & 1.914 & 1.968 & 1.974 & 3.181 & 1.754 & 2.888 & 2.962 & 2.964 \\
			& k     & 20    & 10    & 9     & 8     & 7     & 22    & 11    & 9     & 8     & 7 \\
			512   & Max. error & 2.38E-07 & 7.81E-08 & 1.29E-07 & 1.52E-07 & 1.65E-07 & 1.03E-09 & 1.72E-10 & 3.43E-10 & 3.94E-10 & 4.14E-10 \\
			& Rate  & 2.211 & 1.679 & 1.955 & 1.983 & 1.984 & 3.194 & 2.331 & 2.944 & 2.981 & 2.979 \\
			& k     & 20    & 10    & 9     & 8     & 7     & 22    & 10    & 9     & 8     & 7 \\
			1024  & Max. error & 5.08E-08 & 2.23E-08 & 3.27E-08 & 3.82E-08 & 4.17E-08 & 1.12E-10 & 2.71E-11 & 4.37E-11 & 4.96E-11 & 5.23E-11 \\
			& Rate  & 2.228 & 1.810 & 1.975 & 1.990 & 1.989 & 3.205 & 2.669 & 2.971 & 2.990 & 2.986 \\
			& k     & 20    & 10    & 9     & 8     & 7     & 22    & 10    & 9     & 8     & 7 \\
			2048  & Max. error & 1.07E-08 & 6.05E-09 & 8.26E-09 & 9.59E-09 & 1.05E-08 & 1.20E-11 & 3.90E-12 & 5.50E-12 & 6.25E-12 & 6.58E-12 \\
			& Rate  & 2.246 & 1.880 & 1.986 & 1.994 & 1.991 & 3.215 & 2.798 & 2.990 & 2.990 & 2.991 \\
			& k     & 20    & 10    & 9     & 8     & 7     & 22    & 11    & 9     & 8     & 7 \\
			\bottomrule
		\end{tabular}%
	\end{adjustbox}
	\caption{Pointwise absolute errors in the maximum norm and computed rates of convergence obtained by the proposed schemes with various fractional orders versus the number of subintervals in Example \ref{ex.3}. We set $s_0=0.2$ for both shooting techniques but $Tol = 10^{-15} (10^{-16})$ for Newton's method (Halley's method), respectively.}
	\label{TABLE-ex3-alpha}
\end{table}%

\subsection{Example \ref{ex.4}: Linear FBPV}
\begin{table}[H]
	\centering
	\centering\small\addtolength{\tabcolsep}{-2.0pt}
	\begin{tabular}{c|cc|cc|cc|cc|cc}
		\toprule
		\multicolumn{11}{c}{Modified integral discretization scheme \cite{Huang2}}\\
		\hline
		& \multicolumn{2}{c|}{$\alpha_2 = 1.1$}       & \multicolumn{2}{c|}{$\alpha_2 = 1.3$}       & \multicolumn{2}{c|}{$\alpha_2 = 1.5$}      & \multicolumn{2}{c|}{$\alpha_2 = 1.7$}       & \multicolumn{2}{c}{$\alpha_2 = 1.9$}  \\
		\hline
		\multicolumn{1}{c}{N} & \multicolumn{1}{c}{Max. error} & \multicolumn{1}{c|}{Rate} & \multicolumn{1}{c}{Max. error} & \multicolumn{1}{c|}{Rate} & \multicolumn{1}{c}{Max. error} & \multicolumn{1}{c|}{Rate} & \multicolumn{1}{c}{Max. error} & \multicolumn{1}{c|}{Rate} & \multicolumn{1}{c}{Max. error} & \multicolumn{1}{c}{Rate} \\
		\hline
		64    & 9.56E-03 & -     & 5.14E-03 & -     & 3.91E-03 & -     & 3.76E-03 & -     & 4.07E-03 & -     \\
		128   & 3.89E-03 & 1.298 & 1.60E-03 & 1.686 & 1.05E-03 & 1.895 & 9.44E-04 & 1.995 & 1.02E-03 & 1.996 \\
		256   & 1.65E-03 & 1.236 & 5.20E-04 & 1.618 & 2.87E-04 & 1.872 & 2.37E-04 & 1.997 & 2.55E-04 & 1.998 \\
		512   & 7.22E-04 & 1.193 & 1.78E-04 & 1.548 & 8.01E-05 & 1.840 & 5.92E-05 & 1.998 & 6.39E-05 & 1.999 \\
		1024  & 3.22E-04 & 1.165 & 6.36E-05 & 1.484 & 2.30E-05 & 1.803 & 1.51E-05 & 1.971 & 1.60E-05 & 1.999 \\
		2048  & 1.45E-04 & 1.148 & 2.36E-05 & 1.432 & 6.77E-06 & 1.762 & 3.90E-06 & 1.955 & 4.00E-06 & 2.000 \\
		\midrule
		\multicolumn{11}{c}{Newton's method}\\
		\hline
		& \multicolumn{2}{c|}{$\alpha_2 = 1.1$}       & \multicolumn{2}{c|}{$\alpha_2 = 1.3$}       & \multicolumn{2}{c|}{$\alpha_2 = 1.5$}      & \multicolumn{2}{c|}{$\alpha_2 = 1.7$}       & \multicolumn{2}{c}{$\alpha_2 = 1.9$}  \\
		\hline
		\multicolumn{1}{c}{N} & \multicolumn{1}{c}{Max. error} & \multicolumn{1}{c|}{Rate} & \multicolumn{1}{c}{Max. error} & \multicolumn{1}{c|}{Rate} & \multicolumn{1}{c}{Max. error} & \multicolumn{1}{c|}{Rate} & \multicolumn{1}{c}{Max. error} & \multicolumn{1}{c|}{Rate} & \multicolumn{1}{c}{Max. error} & \multicolumn{1}{c}{Rate} \\
		\hline
		64    & 4.42E-03 & -     & 1.31E-03 & -     & 1.14E-03 & -     & 1.27E-03 & -     & 1.42E-03 & -     \\
		128   & 1.89E-03 & 1.225 & 3.83E-04 & 1.775 & 2.95E-04 & 1.955 & 3.17E-04 & 1.999 & 3.55E-04 & 2.000 \\
		256   & 8.10E-04 & 1.224 & 1.27E-04 & 1.588 & 7.62E-05 & 1.955 & 7.92E-05 & 2.002 & 8.87E-05 & 2.000 \\
		512   & 3.48E-04 & 1.219 & 6.31E-05 & 1.012 & 1.97E-05 & 1.949 & 1.98E-05 & 2.004 & 2.22E-05 & 2.000 \\
		1024  & 1.50E-04 & 1.213 & 2.88E-05 & 1.130 & 5.15E-06 & 1.938 & 4.92E-06 & 2.006 & 5.54E-06 & 2.000 \\
		2048  & 6.91E-05 & 1.120 & 1.26E-05 & 1.191 & 1.36E-06 & 1.922 & 1.22E-06 & 2.008 & 1.38E-06 & 2.001 \\
		\hline
		\bottomrule
	\end{tabular}%
	\caption{Pointwise absolute errors and computed rates of convergence obtained by the modified integral discretization scheme \cite{Huang2} and Newton's method. We set $s_0=0.2, Tol=10^{-6}$.}
	\label{TABLE-ex4}
\end{table}%

\subsection{Example \ref{ex.5}: Linear FBVP}
\begin{table}[H]
	\centering
	\centering\small\addtolength{\tabcolsep}{-4.9pt}
	\begin{tabular}{c|ccc|ccc|ccc|ccc|ccc}
		\toprule
		\multicolumn{16}{c}{modified integral discretization scheme \cite{Huang2}}\\
		\hline
		& \multicolumn{3}{c|}{$\alpha_2 = 1.1~$}       & \multicolumn{3}{c|}{$\alpha_2 = 1.3$}       & \multicolumn{3}{c|}{$\alpha_2 = 1.5$}      & \multicolumn{3}{c|}{$\alpha_2 = 1.7$}       & \multicolumn{3}{c}{$\alpha_2 = 1.9$}  \\
		\hline
		& \multicolumn{1}{c}{Max. error} & \multicolumn{1}{c}{Rate} & \multicolumn{1}{c|}{Time} & \multicolumn{1}{c}{Max. error} & \multicolumn{1}{c}{Rate} & \multicolumn{1}{c|}{Time} & \multicolumn{1}{c}{Max. error} & \multicolumn{1}{c}{Rate} & \multicolumn{1}{c|}{Time} & \multicolumn{1}{c}{Max. error} & \multicolumn{1}{c}{Rate} & \multicolumn{1}{c|}{Time} & \multicolumn{1}{c}{Max. error} & \multicolumn{1}{c}{Rate} & \multicolumn{1}{c}{Time} \\
		\hline
		10    & 1.93E-03 & -      & 0.02 & 1.67E-03 & -      & 0.02 & 1.43E-03 & -      & 0.02 & 1.29E-03 & -      & 0.02 & 1.21E-03 & -      & 0.02 \\
		20    & 5.48E-04 & 1.816  & 0.02 & 4.68E-04 & 1.831  & 0.02 & 3.97E-04 & 1.847  & 0.02 & 3.54E-04 & 1.862  & 0.02 & 3.32E-04 & 1.870  & 0.02 \\
		40    & 1.46E-04 & 1.907  & 0.09 & 1.24E-04 & 1.915  & 0.09 & 1.04E-04 & 1.925  & 0.09 & 9.28E-05 & 1.932  & 0.09 & 8.68E-05 & 1.937  & 0.1 \\
		80    & 3.78E-05 & 1.948  & 0.34 & 3.21E-05 & 1.954  & 0.34 & 2.68E-05 & 1.962  & 0.35 & 2.38E-05 & 1.966  & 0.34 & 2.22E-05 & 1.969  & 0.34 \\
		160   & 9.67E-06 & 1.968  & 1.35 & 8.16E-06 & 1.974  & 1.45 & 6.80E-06 & 1.980  & 1.34 & 6.01E-06 & 1.983  & 1.36 & 5.60E-06 & 1.984  & 1.35 \\
		320   & 2.45E-06 & 1.979  & 5.43 & 2.06E-06 & 1.984  & 5.44 & 1.71E-06 & 1.990  & 5.33 & 1.51E-06 & 1.991  & 5.32 & 1.41E-06 & 1.992  & 5.39 \\
		
		\midrule
		\multicolumn{16}{c}{Newton's method}\\
		\hline
		& \multicolumn{3}{c|}{$\alpha_2 = 1.1~$}       & \multicolumn{3}{c|}{$\alpha_2 = 1.3$}       & \multicolumn{3}{c|}{$\alpha_2 = 1.5$}      & \multicolumn{3}{c|}{$\alpha_2 = 1.7$}       & \multicolumn{3}{c}{$\alpha_2 = 1.9$}  \\
		\hline
		& \multicolumn{1}{c}{Max. error} & \multicolumn{1}{c}{Rate} & \multicolumn{1}{c|}{Time} & \multicolumn{1}{c}{Max. error} & \multicolumn{1}{c}{Rate} & \multicolumn{1}{c|}{Time} & \multicolumn{1}{c}{Max. error} & \multicolumn{1}{c}{Rate} & \multicolumn{1}{c|}{Time} & \multicolumn{1}{c}{Max. error} & \multicolumn{1}{c}{Rate} & \multicolumn{1}{c|}{Time} & \multicolumn{1}{c}{Max. error} & \multicolumn{1}{c}{Rate} & \multicolumn{1}{c}{Time} \\
		\hline
		10    & 7.27E-03 & -      & 0.01 & 1.97E-03 & -      & 0.02 & 9.45E-04 & -      & 0.01 & 6.80E-04 & -      & 0.02 & 5.80E-04 & -      & 0.19 \\
		20    & 1.78E-03 & 2.033  & 0.02 & 3.92E-04 & 2.329  & 0.02 & 1.97E-04 & 2.263  & 0.02 & 1.56E-04 & 2.124  & 0.02 & 1.40E-04 & 2.047  & 0.02 \\
		40    & 3.92E-04 & 2.180  & 0.07 & 7.80E-05 & 2.328  & 0.07 & 4.38E-05 & 2.170  & 0.06 & 3.75E-05 & 2.055  & 0.07 & 3.47E-05 & 2.016  & 0.08 \\
		80    & 8.38E-05 & 2.226  & 0.24 & 1.62E-05 & 2.272  & 0.24 & 1.03E-05 & 2.093  & 0.23 & 9.25E-06 & 2.021  & 0.24 & 8.64E-06 & 2.005  & 0.24 \\
		160   & 1.78E-05 & 2.236  & 0.93 & 3.50E-06 & 2.205  & 0.93 & 2.49E-06 & 2.046  & 0.89 & 2.30E-06 & 2.007  & 0.94 & 2.16E-06 & 2.001  & 0.93 \\
		320   & 3.79E-06 & 2.232  & 3.58 & 7.93E-07 & 2.144  & 3.57 & 6.12E-07 & 2.021  & 3.55 & 5.74E-07 & 2.002  & 3.62 & 5.39E-07 & 2.000  & 3.56 \\
		\midrule
		\multicolumn{16}{c}{Halley's method}\\
		\hline
		& \multicolumn{3}{c|}{$\alpha_2 = 1.1~$}       & \multicolumn{3}{c|}{$\alpha_2 = 1.3$}       & \multicolumn{3}{c|}{$\alpha_2 = 1.5$}      & \multicolumn{3}{c|}{$\alpha_2 = 1.7$}       & \multicolumn{3}{c}{$\alpha_2 = 1.9$}  \\
		\hline
		& \multicolumn{1}{c}{Max. error} & \multicolumn{1}{c}{Rate} & \multicolumn{1}{c|}{Time} & \multicolumn{1}{c}{Max. error} & \multicolumn{1}{c}{Rate} & \multicolumn{1}{c|}{Time} & \multicolumn{1}{c}{Max. error} & \multicolumn{1}{c}{Rate} & \multicolumn{1}{c|}{Time} & \multicolumn{1}{c}{Max. error} & \multicolumn{1}{c}{Rate} & \multicolumn{1}{c|}{Time} & \multicolumn{1}{c}{Max. error} & \multicolumn{1}{c}{Rate} & \multicolumn{1}{c}{Time} \\
		\hline
		10    & 1.40E-03 & -      & 0.02 & 2.79E-04 & -     & 0.02 & 1.08E-04 & -      & 0.02 & 6.95E-05 & -      & 0.02 & 5.51E-05 & -      & 0.02 \\
		20    & 1.55E-04 & 3.169  & 0.03 & 2.57E-05 & 3.441 & 0.03 & 1.06E-05 & 3.339  & 0.03 & 7.70E-06 & 3.174  & 0.03 & 6.56E-06 & 3.070  & 0.03 \\
		40    & 1.65E-05 & 3.237  & 0.11 & 2.44E-06 & 3.397 & 0.11 & 1.15E-06 & 3.211  & 0.11 & 9.19E-07 & 3.066  & 0.11 & 8.12E-07 & 3.014  & 0.11 \\
		80    & 1.73E-06 & 3.253  & 0.37 & 2.43E-07 & 3.325 & 0.40 & 1.33E-07 & 3.111  & 0.37 & 1.13E-07 & 3.023  & 0.37 & 1.02E-07 & 3.000  & 0.37 \\
		160   & 1.82E-07 & 3.252  & 1.40 & 2.56E-08 & 3.248 & 1.42 & 1.60E-08 & 3.055  & 1.38 & 1.41E-08 & 3.007  & 1.41 & 1.27E-08 & 2.998  & 1.40  \\
		320   & 1.92E-08 & 3.244  & 5.47 & 2.82E-09 & 3.180 & 5.44 & 1.96E-09 & 3.026  & 5.71 & 1.76E-09 & 3.002  & 5.44 & 1.59E-09 & 2.999  & 5.45 \\
		
		\hline
		\bottomrule
	\end{tabular}%
	\caption{Pointwise absolute errors and computed rates of convergence obtained by the modified integral discretization scheme \cite{Huang2} and proposed methods: Newton's method, Halley's method. We set $s_0=0.2, Tol=10^{-5}$.}
	\label{TABLE-ex5}
\end{table}%



\section{Linear Explicit Method}\label{apndx3}
Let us consider the following linear single-term FBVP with RBCs:
\begin{equation}
	\begin{cases}\label{direct-eqn1}
		D_{0}^{\alpha_2}y(t) = f(t)+c(t)y(t)+b(t)y'(t), \\
		a_1y(0)+b_1y'(0) = \gamma_1,~~a_2y(1)+b_2y'(1) = \gamma_2,
	\end{cases}
\end{equation}
where $1 < \alpha_2 < 2$.
By Lemmas \ref{Lemma_of_Gronwall_ineq} through \ref{order_com_lemma} and Theorem \ref{epsapprox}, the FBVP \eqref{direct-eqn1} is equivalent to the following system
\begin{equation}\label{direct-eqn2}
	\begin{cases}
		D_{0}^{1-\epsilon}y(t) = z(t) & y(0) = s,\\
		D_{0}^{\alpha}z(t) = f(t) + c(t)y(t) + b(t)z(t) & z(0) = (\gamma_1-a_1s)/b_1.
	\end{cases}
\end{equation}
Expressing the solution of \eqref{direct-eqn2} as the discretized form of the Volterra integral equation which is equivalent to \eqref{direct-eqn2}, we obtain
\begin{equation}\label{direct-eqn3}
	\begin{aligned}
		y(t_{n+1}) &= y(0) + \frac{1}{\Gamma(1-\epsilon)}\int_{t_0}^{t_{n+1}}(t_{n+1}-\tau)^{-\epsilon}z(\tau)d\tau,\\
		z(t_{n+1}) &= z(0)+J_0^{\alpha_2-1}f(t_{n+1})+\frac{1}{\Gamma(\alpha_2)}\int_{t_0}^{t_{n+1}}(t_{n+1}-\tau)^{\alpha_2-1}\Bigl{(} c(\tau)y(\tau)+b(\tau)z(\tau) \Bigr{)} d\tau.
	\end{aligned}
\end{equation}
The approximation solutions to $y(t_{n+1})$ and $z(t_{n+1})$ in \eqref{direct-eqn3} with $s_k$ can be explicitly described as follows:
\begin{equation*}
	\begin{aligned}
		y_{n+1} &= s_k + \frac{1}{\Gamma(1-\epsilon)}\sum_{j=0}^{n}\int_{t_{j}}^{t_{j+1}}(t_{n+1}-\tau)^{-\epsilon}z(\tau)d\tau,\\
		z_{n+1} &= \frac{\gamma_1-a_1s_k}{b_1}+J_0^{\alpha_2-1}f(t_{n+1})+\frac{1}{\Gamma(\alpha_2)}\sum_{j=0}^{n}\int_{t_j}^{t_{j+1}}(t_{n+1}-\tau)^{\alpha_2-1}\Bigl{(} c(\tau)y(\tau)+b(\tau)z(\tau) \Bigr{)} d\tau.
	\end{aligned}
\end{equation*}
Replacing $y(\tau)$ and $z(\tau)$ with linear interpolation, we have
\begin{equation}\label{direct-eqn4}
	\begin{aligned}
		y_{n+1} &= s_k + \frac{1}{\Gamma(1-\epsilon)}\sum_{j=0}^{n}\int_{t_{j}}^{t_{j+1}}(t_{n+1}-\tau)^{-\epsilon}\Bigl{(} \frac{t_{j+1}-\tau}{h}z_j+\frac{t_j-\tau}{-h}z_{j+1} \Bigr{)}d\tau,\\
		z_{n+1} &= \frac{\gamma_1-a_1s_k}{b_1} + J_0^{\alpha_2-1}f(t_{n+1})+\\
		&\hspace{1cm}\frac{1}{\Gamma(\alpha_2)}\sum_{j=0}^{n}\int_{t_j}^{t_{j+1}}(t_{n+1}-\tau)^{\alpha_2-1}\left\{ \frac{t_{j+1}-\tau}{h}\Bigl{(}c(t_j)y_j+b(t_j)z_j \Bigr{)} +\frac{t_j-\tau}{-h}\Bigl{(}c(t_{j+1})y_{j+1}+b(t_{j+1})z_{j+1} \Bigr{)}\right\} d\tau.
	\end{aligned}
\end{equation}
Let us shorten the expression of $y_{n+1}, z_{n+1}$ in \eqref{direct-eqn4} as follows:
\begin{equation*}
	\begin{aligned}
		y_{n+1} &= s_k + \sum_{j=0}^{n}\Bigr{[} A^1_jz_j+A^2_jz_{j+1} \Bigr{]},\\
		z_{n+1} &= \frac{\gamma_1-a_1s_k}{b_1}+J_0^{\alpha_2-1}f(t_{n+1})+\sum_{j=0}^{n-1}\Bigr{[} B^1_j\Bigl{(}c(t_j)y_j+b(t_j)z_j \Bigr{)} + B^2_j\Bigl{(}c(t_{j+1})y_{j+1}+b(t_{j+1})z_{j+1} \Bigr{)}\Bigr{]}\\
		&\hspace{1cm} + B^1_n\Bigl{(}c(t_n)y_n+b(t_n)z_n \Bigr{)} + B^2_n\Bigl{(}c(t_{n+1})y_{n+1}+b(t_{n+1})z_{n+1} \Bigr{)}.
	\end{aligned}
\end{equation*}
We omit describing definitions of $A^i_j, B^i_j, \ i=1,2, \ j=0,\ldots, n$ because it is straightforward.
Substituting the explicit form of $y_{n+1}$ into the right-hand side of $z_{n+1}$, we have
\begin{equation}\label{direct-eqn5}
	\begin{aligned}
		z_{n+1} &= \frac{\gamma_1-a_1s_k}{b_1} + J_0^{\alpha_2-1}f(t_{n+1})+\sum_{j=0}^{n-1}\Bigr{[} B^1_j\Bigl{(}c(t_j)y_j+b(t_j)z_j \Bigr{)} + B^2_j\Bigl{(}c(t_{j+1})y_{j+1}+b(t_{j+1})z_{j+1} \Bigr{)}\Bigr{]}\\
		&\hspace{1cm} + B^1_n\Bigl{(}c(t_n)y_n+b(t_n)z_n \Bigr{)} + B^2_n\Bigr{[}c(t_{n+1})\left\{s_k + \sum_{j=0}^{n}\Bigl{(} A^1_jz_j+A^2_jz_{j+1} \Bigr{)} \right\}+b(t_{n+1})z_{n+1} \Bigr{]},\\
		&= \frac{\gamma_1-a_1s_k}{b_1} + J_0^{\alpha_2-1}f(t_{n+1})+\sum_{j=0}^{n-1}\Bigr{[} B^1_j\Bigl{(}c(t_j)y_j+b(t_j)z_j \Bigr{)} + B^2_j\Bigl{(}c(t_{j+1})y_{j+1}+b(t_{j+1})z_{j+1} \Bigr{)}\Bigr{]}+ B^1_n\Bigl{(}c(t_n)y_n+b(t_n)z_n \Bigr{)}\\
		&\hspace{1cm}+ B^2_nc(t_{n+1})\left\{s_k + \sum_{j=0}^{n-1}\Bigl{(} A^1_jz_j+A^2_jz_{j+1} \Bigr{)}+A^1_nz_n \right\}+\Bigl{(}B_n^2A_{n}^2c(t_{n+1})+b(t_{n+1})\Bigr{)}z_{n+1}.
	\end{aligned}
\end{equation}
Since the right-hand side of \eqref{direct-eqn5} is linear in $z_j, j=0,\ldots, n+1$, $z_{n+1}$ can be explicitly expressed as follows:
\begin{equation*}
	\begin{aligned}
		z_{n+1} &= \bigg[ \frac{\gamma_1-a_1s_k}{b_1} + J_0^{\alpha_2-1}f(t_{n+1})+\sum_{j=0}^{n-1}\Bigr{[} B^1_j\Bigl{(}c(t_j)y_j+b(t_j)z_j \Bigr{)} + B^2_j\Bigl{(}c(t_{j+1})y_{j+1}+b(t_{j+1})z_{j+1} \Bigr{)}\Bigr{]}+ B^1_n\Bigl{(}c(t_n)y_n+b(t_n)z_n \Bigr{)}\\
		&\hspace{1cm} + B^2_nc(t_{n+1})\left\{s_k + \sum_{j=0}^{n-1}\Bigl{(} A^1_jz_j+A^2_jz_{j+1} \Bigr{)}+A^1_nz_n \right\}\bigg]\bigg/ \bigg[ 1- \left\{ B_n^2A_{n}^2c(t_{n+1})+b(t_{n+1})\right\} \bigg].
	\end{aligned}
\end{equation*}


\section{References}
\bibliographystyle{elsarticle-num}
\bibliography{reference}

\begin{thebibliography}{10}
\expandafter\ifx\csname url\endcsname\relax
  \def\url#1{\texttt{#1}}\fi
\expandafter\ifx\csname urlprefix\endcsname\relax\def\urlprefix{URL }\fi
\expandafter\ifx\csname href\endcsname\relax
  \def\href#1#2{#2} \def\path#1{#1}\fi

\bibitem{nguyen2017high}
T.~B. Nguyen, B.~Jang, A high-order predictor-corrector method for solving
  nonlinear differential equations of fractional order, Fractional Calculus and
  Applied Analysis 20~(2) (2017) 447--476.

\bibitem{Caputo}
M.~Caputo, Linear model of dissipation whose {Q} is almost frequency
  independent -- {II}, The Geophysical Journal of the Royal Astronomical
  Society 13 (1967) 529--539.

\bibitem{Dalir}
M.~Dalir, M.~Bashour, Applications of fractional calculus, Appl. Math. Sci. 4
  (2010) 1021--1032.

\bibitem{Das}
S.~Das, Functional Fractional Calculus for System Identification and Controls,
  Springer, New York, 2008.

\bibitem{Hilfer}
R.~Hilfer, Applications of Fractional Calculus in Physics, World Scientific,
  Singapore, 2000.

\bibitem{Magin}
R.~Magin, Fractional Calculus in Bioengineering, Begell House Publishers, 2006.

\bibitem{Podlubny}
I.~Podlubny, Fractional differential equations, Academic Press, San Diego,
  1999.

\bibitem{diethelm2010analysis}
K.~Diethelm, The analysis of fractional differential equations: An
  application-oriented exposition using differential operators of Caputo type,
  Springer Science \& Business Media, 2010.

\bibitem{Huang3}
Z.~Cen, J.~Huang, A.~Le, A modified integral discretization scheme for a
  two-point boundary value problem with a caputo fractional derivative, J.
  Comput. Appl. Math. 367~(15) (2020) 112465.

\bibitem{Huang2}
Z.~Cen, J.~Huang, A.~Xu, An efficient numerical method for a two-point boundary
  value problem with a caputo fractional derivative an efficient numerical
  method for a two-point boundary value problem with a caputo fractional
  derivative, J. Comput. Appl. Math. 336 (2018) 1--7.

\end{thebibliography}

\end{document}